\documentclass[11pt]{article}
\usepackage{color}
\usepackage{amsmath}
\catcode`\@=11 \@addtoreset{equation}{section}

\catcode`\@=12

\usepackage{amssymb}
\usepackage{amsfonts}
\usepackage{xcolor}
\usepackage{graphics}
\usepackage{epsfig,psfrag,graphicx}
\usepackage{amssymb}
\usepackage{eepic,epic}
\usepackage{dsfont}

\usepackage{enumerate}

\usepackage{epsfig} 
\usepackage{amsmath} 
\usepackage{amssymb} 
\usepackage{amsbsy} 

\newtheorem{definition}{Definition}[section]
\newtheorem{theorem}[definition]{Theorem}
\newtheorem{lemma}[definition]{Lemma}

\newtheorem{proposition}[definition]{Proposition}
\newtheorem{remark}[definition]{Remark}

\newcommand{\nc}{\newcommand}
\nc{\qed}{\mbox{}\nolinebreak\hfill \rule{2mm}{2mm}} 
\nc{\weak}{\rightharpoonup}
\nc{\weakstar}{\stackrel{\ast}{\rightharpoonup}} 
\nc{\proof}{{\bf Proof: }} 
\renewcommand{\div}{{{\mathrm{div}}_x}\,}
\newcommand{\vrho}{\varrho}
\nc{\modular}[1]{{\stackrel{ #1}{\longrightarrow\,}}}

\def\bbbone{{\mathchoice {\rm 1\mskip-4mu l}
{\rm 1\mskip-4mu l} {\rm 1\mskip-4.5mu l} {\rm 1\mskip-5mu l}}}

\def\tens#1{\pmb{\mathsf{#1}}}
\def\vec#1{\boldsymbol{#1}}

\newcommand{\vu}		{\vec{u}}
\newcommand{\vr}		{\vrho}
\newcommand{\vrm}       {\vrho^{(1)}_{\veps,M}}
\newcommand{\vre}		{\vr_\varepsilon}
\newcommand{\vrez}	{\vr_{0,\varepsilon}}

\newcommand{\vret}		{\wtilde{\vr}_\ep}
\newcommand{\ue}		{\vec{u}_\varepsilon}
\newcommand{\uez}		{\vec{u}_{0,\varepsilon}}
\newcommand{\ep}		{\varepsilon}

\newcommand{\n}		{\vec{n}}
\newcommand{\vU}		{\vec{U}}
\newcommand{\ess} 	{{\rm{ess}}}
\newcommand{\res}		{{\rm{res}}}
\newcommand{\dx}		{\,{\rm d}x}
\newcommand{\dt}		{\, {\rm d}t}
\newcommand{\dxdt}	{\, {\rm d}x{\rm d}t}

\newcommand{\U}		{\vec{U}}

\def\bbbone{{\mathchoice {\rm 1\mskip-4mu l}
{\rm 1\mskip-4mu l} {\rm 1\mskip-4.5mu l} {\rm 1\mskip-5mu l}}}

\renewcommand{\bbbone}{\mathds{1}}

\newcommand{\tbf}{\textbf}

\newcommand{\tsl}{\textsl}

\newcommand{\mbb}{\mathbb}

\newcommand{\mc}{\mathcal}
\newcommand{\mf}{\mathfrak}
\newcommand{\veps}{\varepsilon}

\newcommand{\what}{\widehat}
\newcommand{\wtilde}{\widetilde}
\newcommand{\vphi}{\varphi}
\newcommand{\oline}{\overline}

\newcommand{\ra}{\rightarrow}

\newcommand{\g}{\gamma}

\newcommand{\s}{\sigma}

\newcommand{\de}{\delta}

\newcommand{\lan}{\langle}
\newcommand{\ran}{\rangle}

\newcommand{\e}{\vec{e}}

\newcommand{\R}{\mathbb{R}}

\newcommand{\N}{\mathbb{N}}
\newcommand{\Z}{\mathbb{Z}}

\newcommand{\T}{\mathbb{T}^1}

\newcommand{\TT}{\mathbb{T}}

\newcommand{\h}{\mathbb{H}}

\renewcommand{\div}{{\rm div}\,}
\newcommand{\curl}{{\rm curl}\,}
\newcommand{\divh}{{\rm div}_h}
\newcommand{\curlh}{{\rm curl}_h}

\newcommand{\Id}{{\rm Id}\,}
\newcommand{\Supp}{{\rm Supp}\,}

\allowdisplaybreaks

\def\d{\partial}
\def\div{{\rm div}\,}

\title{\LARGE \bf{On the influence of gravity 
in the dynamics of \\ geophysical flows
}}
\author{ \textsl{Daniele Del Santo}$\,^1\;$, $\;$\textsl{Francesco Fanelli}$\,^2\;$, $\;$\textsl{Gabriele Sbaiz}$\,^{1,2}\;$,
$\;$\textsl{Aneta Wr\'oblewska-Kami\'nska}$\,^{3}$ 
\vspace{.2cm} \\
\footnotesize{$\,^1\;$ \textsc{Universit\`a degli Studi di Trieste}, \textit{Dipartimento di Matematica e Geoscienze},} \\
\footnotesize{Via Valerio 12/1, 34127 Trieste, Italy} \vspace{0.2cm} \\
\footnotesize{$\,^2\;$ \textsc{Univ. Lyon, Universit\'e Claude Bernard Lyon 1}, CNRS UMR 5208, \textit{Institut Camille Jordan},} \\
\footnotesize{43 blvd. du 11 novembre 1918, F-69622 Villeurbanne cedex, France} \vspace{0.2cm} \\
\footnotesize{$\,^3\;$ \textsc{Institute of Mathematics of Polish Academy of Sciences},} \\ {\footnotesize ul.\'Sniadeckich 8, 00-656 Warszawa, Poland}
\vspace{.3cm} \\
\footnotesize{\ttfamily{delsanto@units.it}$\,,\quad$  \ttfamily{fanelli@math.univ-lyon1.fr}$\,,\quad$\ttfamily{gabriele.sbaiz@phd.units.it}$\,,\quad$
\ttfamily{a.wroblewska@impan.pl}
} \vspace{.1cm}
}

\date{\small \today}

\begin{document}
\maketitle

\abstract{In the present paper, we study a multiscale limit for the barotropic Navier-Stokes system with Coriolis and gravitational forces,
for vanishing values of the Mach, Rossby and Froude numbers ($\rm Ma$, $\rm Ro$ and $\rm Fr$, respectively).
The focus here is on the effects of gravity: albeit remaining in a low stratification regime ${\rm Ma}/{\rm Fr}\,\ra\,0$,
we consider scaling for the Froude number which go beyond the ``critical'' value $\rm Fr\,=\,\sqrt{\rm Ma}$.
The rigorous derivation of suitable limiting systems for the various choices of the scaling is shown by means of a compensated compactness argument.
Exploiting the precise structure of the gravitational force is the key to get the convergence.
}

\paragraph*{\small 2020 Mathematics Subject Classification:}{\footnotesize 35Q86 
(primary);
35B25, 
76U60, 
35B40, 
76M45 
(secondary).}


\section{Introduction}

In this paper we continue the investigation we began in \cite{DS-F-S-WK}, about multiscale analysis of mathematical models for geophysical flows.
Our focus here is on the effect of gravity in regimes of \emph{low stratification}, but which go beyond a choice of the scaling that, in light
of previous results, we call ``critical''.

In order to explain better all this, let us introduce some physics about the problem we are interested in, and give an overview of related studies.
We present in Section \ref{s:result} the precise system we will work on, and the statements of our main results.

\subsection{Some physical considerations}

By definition (see \tsl{e.g.} \cite{C-R}, \cite{Ped}), geophysical flows are flows whose dynamics is characterised by large time and space scales.
Typical examples are currents in the atmosphere and the ocean, but of course there are many other cases where such fluids occur out of the Earth, like flows on stars or other celestial bodies.

At those scales, the effects of the rotation of the ambient space (which in the previous examples is the Earth) are no more negligible, and the fluid motion undergoes the action of a strong Coriolis force.
A simplistic assumption, which is however often adopted in physical and mathematical studies, consists in restricting the attention to flows at mid-latitudes, \tsl{i.e.} flows which take place far enough from the poles and the equatorial zone. 
Thus, if we denote by $\vr\,\geq\,0$ the density of the fluid and by $\vu\,\in\,\R^3$ its velocity field, the Coriolis force may be represented in the following form:
\begin{equation} \label{def:Coriolis}
\mf C(\vr,\vu)\,:=\,\frac{1}{\rm Ro}\,\vec e_3\times\vr\,\vu\,,
\end{equation}
where $\vec e_3=(0,0,1)$, 
the symbol $\times$ denotes the classical external product of vectors in $\R^3$ and $\rm Ro>0$ is the so-called \emph{Rossby number},
a physical adimensional parameter linked to the speed of rotation of the Earth. 
In particular, the previous definition implies that the rotation is approximated to take place around the vertical axis, and its strength does not depend
on the latitude. We point out that, despite all these simplifications, the obtained model is already able to give a quite accurate description of several physically relevant phenomena occurring in the dynamics of geophysical fluids (see \tsl{e.g.} \cite{Ped}, \cite{C-D-G-G}).

In geophysical fluid dynamics, effects of the fast rotation are predominant; this translates into the fact that the value of $\rm Ro$
is very small.
As a matter of fact, the Rossby number $\rm Ro$ is defined as the ratio between the nonlinear acceleration to the Coriolis parameter term, namely
$$ {\rm Ro}:=\frac{U_{\rm ref}}{f_{\rm ref}\,L_{\rm ref}}\,, $$
where $U_{\rm ref},\, L_{\rm ref}$ and $f_{\rm ref}$ are respectively the horizontal velocity scale, the horizontal length scale and the reference Coriolis frequency (see \tsl{e.g.} \cite{K-C-D} for more details).
For instance, for a typical atmospheric value of $U_{\rm ref} \sim 10$ m/s, $f_{\rm ref}\sim 10^{-4} \text{ s}^{-1}$ and $L_{\rm ref}\sim 1000$ km, the Rossby number
turns out to be $0.1$; its value is even smaller for many flows in the oceans.
As established by the \emph{Taylor-Proudman theorem} in geophysics, the fast rotation imposes a certain rigidity/stability, as it undresses the motion of any vertical variation, and forces it to take place on planes orthogonal to the rotation axis.
Thus, the dynamics becomes purely two-dimensional and horizontal, and the fluid tends to move along vertical columns.

However, such an ideal configuration is hinder by another fundamental force acting at geophysical scales, the gravity, which works to restore vertical stratification of the density. The gravitational force may be represented by the term
\[
 \mc G(\vr)\,:=\,-\,\frac{1}{\rm Fr^2}\,\vr\,\vec e_3\,, 
\]
where $\rm Fr>0$ is the \emph{Froude number}, another physical adimensional parameter, which measures the importance of the stratification effects in the dynamics.
In geophysics (see again \cite{K-C-D} for details), $\rm Fr$ represents the square root of the ratio between inertia and gravity, namely
$$ {\rm Fr}:=\frac{U_{\rm ref}}{\sqrt{g\,L_{\rm ref}}}\,, $$
where $g$ is the acceleration of gravity.

As it happens for the Rossby number,
at large scales also the Froude parameter is typically very small. Thus, the competition between the stabilisation effect of the Coriolis force and the vertical stratification due to gravity, is translated in the model into the competition between the orders of magnitude of the two parameters $\rm Ro$ and $\rm Fr$.

Actually, it turns out that the gravity $\mc G$ acts in combination with pressure forces. Restricting from now on our attention to the case of compressible fluids, like currents in the atmosphere for instance, and neglecting for a while heat transfer processes, the pressure term arising in the mathematical model takes the form
\[
\mf P(\vr)\,:=\,\frac{1}{\rm Ma^2}\,\nabla p(\vr)\,,
\]
where $p$ is a known smooth function of the density (and, in the general case, of the temperature of the fluid) and $\rm Ma>0$ is the so-called \emph{Mach number}, a third fundamental adimensional parameter which sets the size of isentropic departures from incompressible flow: the more $\rm Ma$ is small, the more compressibility effects are low. Also the value of $\rm Ma$ is very small for geophysical flows, since it is defined as
$$ {\rm Ma}:=\frac{U_{\rm ref}}{c}\, $$
with $c$ being the sound speed (for instance, in the oceans the typical sound speed is $c\sim 1520$ m/s).

\medbreak
As it is customary in physical studies, because of the complexity of the model, one would like to derive reduced models for geophysical flows,
which however are able to retain most of the properties of the original system.
The problem is that the three terms $\mf C$, $\mc G$ and $\mf P$ enter into play in the model with a very large prefactor in front of them, owing to the
smallness of the values of $\rm Ro$, $\rm Fr$ and $\rm Ma$ respectively. 
For assessing their relative importance and their influence in the dynamics, one fixes a choice of their orders of magnitude. Actually (see \tsl{e.g.} the discussion in Section 1.4 of \cite{C-R}), there is some arbitrariness in doing so, depending on the specific properties of the physical
system and on the processes one wants to put the accent on.

In general, geophysical fluid dynamics is a multiscale process, meaning that Earth's rotation, gravity and pressure forces act, and interact, at different
scales in the system. In other words, $\rm Ro$, $\rm Fr$ and $\rm Ma$ have different orders of magnitude, and only for some specific choices,
all (or some) of them are in balance.

\subsection{Multiscale analysis: an overview of previous results}

At the mathematical level, in the last 30 years there has been a huge amount of works devoted to the rigorous justification, in various functional frameworks,
of the reduced models considered in geophysics. Studies have been carried out in various contexts: for instance, focusing only on the effects of the low Mach number, or on its interplay with a low Froude number regime.

Reviewing the whole literature about this subject goes far beyond the scopes of this introduction, therefore we make the choice to report only on works which
deal with the presence of the Coriolis force \eqref{def:Coriolis}.
We also decide to leave aside from the discussion the case of incompressible models, because (owing to the rigidity imposed by the divergence-free constraint
on the velocity field of the fluid) less pertinent for multiscale analysis. We refer to book \cite{C-D-G-G} and the references therein for a panorama of
known results for incompressible homogeneous fluids, even though more recent developments have been made (see \tsl{e.g.} \cite{Scrobo} for a case where stratification is considered).
Notice that there are also a few recent works \cite{Fan-G}, \cite{C-F}, \cite{Sbaiz}, dealing with incompressible non-homogeneous fluids, but results in
that direction are only partial and the general picture still remains poorly understood at present.

The framework of compressible fluid models, instead, provides a much richer setting for the multiscale analysis of geophysical flows.
In what we are going to say, we make the choice of focusing on works which deal with viscous flows and which perform the asymptotic study for general
ill-prepared initial data. However, the literature about the subject is more ample than that.

First results in that direction were presented in \cite{F-G-N}, \cite{F-G-GV-N} for the barotropic Navier-Stokes system. There, the authors investigated the
combined effect of a strong Coriolis force (low Rossby number limit) and of weak compressibility of the fluid (low Mach number limit),
under the scaling
\begin{equation} \label{eq:scale}
{\rm Ma}\,=\,\veps^m\qquad\mbox{ and }\qquad {\rm Ro}\,=\,\veps\,,\qquad\qquad \mbox{ with }\quad m\geq1\,,
\end{equation}
where $\veps\in\,]0,1]$ is a small parameter, which one wants to let go to $0$ in order to derive an asymptotic model.
Notice that in \cite{F-G-GV-N} the effects due to the centrifugal force were considered as well, but this imposed the severe restriction $m>10$.
In the case $m=1$ in \eqref{eq:scale}, the system presents an isotropic scaling, since the Rossby and Mach numbers act at the same order of magnitude and they keep in balance in the limit process. This balance takes the name of quasi-geostrophic balance, and the limit system
is identified as the so-called \emph{quasi-geostrophic equation} for the stream function of the target velocity field.
When $m>1$, instead, the pressure and Coriolis forces act at different scales, the former one having a predominant effect on the dynamics of the fluid. At the mathematical level, the anisotropy of scaling generates some complications in the analysis; in \cite{F-G-GV-N} this issue was handled by the use of dispersive
estimates, which allowed to show convergence to a $2$-D incompressible Navier-Stokes system.

We refer to \cite{F_MA} for a similar study in the context of capillary models. There, the choice $m=1$ was made, but the anisotropy was given by the scaling
fixed for the internal forces term (the so-called Korteweg stress tensor). In addition, we refer to \cite{F_2019} for the case of large Mach numbers,
namely for the case when $0\leq m<1$ in \eqref{eq:scale}. Since, in that instance, the pressure gradient is not strong enough to compensate the Coriolis force,
in order to find some interesting dynamics in the limit one has to introduce a penalisation of the bulk viscosity coefficient.

In \cite{F-N_AMPA}, \cite{F-N_CPDE} the effects of gravity were added, under the scaling
\begin{equation} \label{eq:scale-G}
{\rm Fr}\,=\,\veps^n\,,\qquad\qquad\mbox{ with }\quad 1\,\leq\,n\,<\,\frac{m}{2}\,.
\end{equation}
In particular, in those works one had $m>2$. As before, a planar incompressible Navier-Stokes system was identified as the limiting system, but, as already
mentioned, the anisotropy of scaling created several difficulties in the analysis. We refer to \cite{K-M-N} and \cite{K-N} for related studies in the
context of the full Navier-Stokes-Fourier system, under the same choices of the scaling (notice that, in \cite{K-N}, the case $m=1$ was considered, but the gravitational force was not penalised at all).
The asymptotic results of \cite{F-N_AMPA}, \cite{F-N_CPDE}, \cite{K-M-N} and \cite{K-N} are all based on a fine combination of the relative entropy/relative energy method with dispersive estimates derived from oscillatory integrals, and a strong argument which allows to handle the ill-preparation of the data (typically, the use of relative energy estimates requires to consider well-prepared initial data).
In all those works, a vanishing viscosity regime was also considered.

In our recent work \cite{DS-F-S-WK}, devoted to the full Navier-Stokes-Fourier system in presence of stratification, we were able to improve the choice of the scaling \eqref{eq:scale-G} and take the endpoint case $n\,=\,m/2$, with $m\geq 1$ as in \eqref{eq:scale}. In passing, we mention that also effects of
the centrifugal force were considered in \cite{DS-F-S-WK}, but this imposed the additional constraint $m\geq 2$ on the order of the Mach number
(which, besides, refined the restriction in \cite{F-G-GV-N}).
The improvement on the orders of the scaling was possible, essentially due to a different technique employed for proving convergence, based on \emph{compensated compactness} arguments.
We refer to \cite{G-SR_2006} for the first implementation of that method in the context of fast rotating fluids, to \cite{F-G-GV-N}, \cite{F_JMFM},
\cite{F_2019} for other applications in the case of non-homogeneous flows. In particular, the convergence is not quantitative at all, but only qualitative. This technique is purely based on the algebraic structure of the system, which allows to find smallness (and vanishing to the limit) of suitable non-linear quantities, and fundamental compactness properties for other quantities (linked to the vorticity of the fluid and to the variations of the density);
such compactness properties were already put in evidence in \cite{Fan-G} (see also \cite{C-F}) in the context of non-homogeneous incompressible fluids in fast
rotation. All these features were enough to pass to the limit in the primitive system, and derive the limiting dynamics: a $2$-D incompressible Navier-Stokes system when $m>1$, a quasi-geostrophic equation for the stream function of the limit velocity when $m=1$.

An important point of the study performed in \cite{DS-F-S-WK} is that the scaling $n=m/2$ allowed to deduce some stratification effect in the limit. More precisely, although the limit dynamics was horizontal and two-dimensional, as dictated by the Taylor-Proudman theorem, stratification appeared in the functions representing
departures of the density and temperature from the respective equilibria. On the contrary, in previous works like \cite{F-N_AMPA}, \cite{F-N_CPDE}, \cite{K-M-N}, \cite{K-N}, based on the scaling \eqref{eq:scale-G}, stratification effects were completely absent. In this sense, we call
the endpoint case $n=m/2$ ``critical''.

To conclude this part, we mention that all the results quoted so far concern various regimes of \emph{low stratification}, meaning that,
according to the scaling in \eqref{eq:scale}, \eqref{eq:scale-G}, one has
\[
 \frac{\rm Ma}{\rm Fr}\,\longrightarrow\,0\qquad\qquad\mbox{ when }\qquad \veps\,\ra\,0^+\,.
\]
The \emph{strong stratification} regime, namely when the ratio ${\rm Ma}/{\rm Fr}$ is of order $O(1)$, is particularly delicate for fast rotating fluids.
This is in stark contrast with the results available about the derivation of the anelastic approximation, where rotation is neglected:
we refer \tsl{e.g.} to \cite{Masm}, \cite{BGL}, \cite{F-K-N-Z} and, more recently, \cite{F-Z} (see also \cite{F-N} and references therein for a
more detailed account of previous works). The reason for that has to be ascribed exactly to the competition between vertical stratification (due to gravity) and horizontal stability (which the Coriolis force tends to impose): in the strong stratification regime, vertical oscillations of the solution
(seem to) persist in the limit, and the available techniques do not allow at present to deal with this problem in its full generality.
Nonetheless, partial results have been obtained in the case of well-prepared initial data, by means of a relative entropy method: we refer to \cite{F-L-N}
for the first result, where the mean motion is derived, and to \cite{B-F-P} for an analysis of Ekman boundary layers in that framework.

\subsection{A short overview of the contents of the paper}

In the present work, we continue our investigation from \cite{DS-F-S-WK}, devoted to the multiscale analysis of systems for geophysical fluids and the derivation
of reduced models.

For clarity of exposition, we neglect here heat transfer processes in the fluid, and focus on the classical barotropic Navier-Stokes system
as the primitive system; the more general case of the Navier-Stokes-Fourier system can be handled at the price of some additional technicalities (as done in \cite{DS-F-S-WK}). Also, we simplify the model by neglecting effects due to the centrifugal force. On the one hand, this choice is not dramatic from the physical viewpoint (see the discussion in \cite{C-R}, for instance); on the other hand, we could include the presence of the centrifugal force, after imposing
some additional restrictions on the order of magnitude of the Mach number. We refer to Section \ref{s:result} below for the presentation of the precise equations
we are going to consider in this paper.

We work in the context of global in time \emph{finite energy} weak solutions to the barotropic Navier-Stokes system with Coriolis force, which provides
a good setting for studying singular limits for that system. We consider the general case of \emph{ill-prepared} initial data.

Our goal here is to go beyond the ``critical'' choice ${\rm Fr}\,=\,\sqrt{\rm Ma}$ performed in \cite{DS-F-S-WK}, and investigate
other regimes where the stratification has an even more important effect.
More precisely, we fix the following choice for the parameters $m$ and $n$ appearing in \eqref{eq:scale} and \eqref{eq:scale-G}: we assume that
\begin{equation} \label{eq:scale-our}
\mbox{ either }\qquad m\,>\,1\quad\mbox{ and }\quad m\,<\,2\,n\,\leq\,m+1\,,\qquad\qquad\mbox{ or }\qquad
m\,=\,1\quad\mbox{ and }\quad \frac{1}{2}\,<\,n\,<\,1\,.
\end{equation}
The restriction $n<1$ when $m=1$ is imposed in order to avoid a strong stratification regime: as already mentioned before, it is not clear
at present how to deal with this case for general ill-prepared initial data, as all the available techniques seem to break down in that case.
On the other hand, the restriction $2\,n\leq m+1$ (for $m>1$) looks to be of technical nature. However, it comes out naturally
in at least two points of our analysis, and it is not clear to us if, and how, it is possible to bypass it and consider the remaining range of values
$(m+1)/2\,<\,n\,<\,m$. 
Let us point out that, in our considerations, the relation $n<m$ holds always true,
so we will always work in a low stratification regime.

At the qualitative level, our main results will be quite similar to the ones presented in \cite{DS-F-S-WK}, in particular the limit dynamics will be the same
(after distinguishing the two cases $m>1$ and $m= 1$). We refer again to Section \ref{s:result} for the precise statements.
In this paper, the main point we put the accent on is how using in a fine way not only the structure of the system, but also the precise structure of each term in order to pass to the limit. To be more precise, the fact of considering
values of $n$ going above the threshold $2n=m$ is made possible thanks to special algebraic cancellations involving the gravity term in the system of wave equations.
Such cancellations owe very much to the peculiar form of the gravitational force, which depends on the vertical variable only, and they do not appear, in general, if one wants to consider the action of different forces on the system. As one may easily guess, the case $2n=m+1$ is more involved: indeed, this choice
of the scaling implies the presence of an additional bilinear term of order $O(1)$ in the computations; in turn, this term might not vanish in the limit, differently to what happens in the case $2n<m+1$. In order to see that this does not occur, and that this term indeed disappears in the limit process, one has to use more thoroughly the structure of the system to control the oscillations.

\medbreak
Before moving on, let us give an outline of the paper.
In Section \ref{s:result} we collect our assumptions and we state our main results.
In Section \ref{s:energy} we show the main consequences of the finite energy condition on the family of weak solutions we are going to consider.
Namely, we derive uniform bounds in suitable norms, which allow us to extract weak-limit points, and we explore the constraints those limit points have to satisfy. In Sections \ref{s:proof} and \ref{s:proof-1}, we complete the proof of our main results, showing convergence in the weak formulation of the equations
in the cases $m>1$ and $m=1$, respectively, \tsl{via} a compensated compactness argument.
We conclude the paper with Appendix \ref{app:LP}, where we present some tools from Littlewood-Paley decomposition,
which have been needed in our analysis.

\paragraph*{Some notation and conventions.} 

Let $B\subset\R^n$. The symbol $\bbbone_B$ denotes the characteristic function of $B$.
The notation $C_c^\infty (B)$ stands for the space of $C^\infty$ functions on $\R^n$ and having compact support in $B$. The dual space $\mc D^{\prime}(B)$ is the space of
distributions on $B$. 

Given $p\in[1,+\infty]$, by $L^p(B)$ we mean the classical space of Lebesgue measurable functions $g$, where $|g|^p$ is integrable over $B$ (with the usual modifications for the case $p=+\infty$).
Sometimes, given $T>0$ and $(p,q)\in[1,+\infty]^2$, we use the symbol $L_T^p(L^q)$ to denote the space $L^p\big(0,T;L^q(B)\big)$.
Given $k \geq 0$, we denote by $W^{k,p}(B)$ the Sobolev space of functions which belongs to $L^p(B)$ together with all their derivatives up to order $k$. When $p=2$, we set $W^{k,2}(B)=H^k(B)$.
For the sake of simplicity, we will often omit from the notation the set $B$, that we will explicitly point out if needed.

In the whole paper, the symbols $c$ and $C$ will denote generic multiplicative constants, which may change from line to line, and which do not depend on the small parameter $\veps$.
Sometimes, we will explicitly point out the quantities on which these constants depend, by putting them inside brackets.

Let $\big(f_\veps\big)_{0<\veps\leq1}$ be a sequence of functions in a normed space $Y$. If this sequence is bounded in $Y$,  we use the notation $\big(f_\veps\big)_{\veps} \subset Y$.
 
\medbreak
Next, let us introduce some notation specific to fluids in fast rotation.

If $B$ is a domain in $\R^3$, we decompose $x\in B$
into $x=(x^h,x^3)$, with $x^h\in\R^2$ denoting its horizontal component. Analogously,
for a vector-field $v=(v^1,v^2,v^3)\in\R^3$, we set $v^h=(v^1,v^2)$ and we define the differential operators
$\nabla_h$ and $\div_{\!h}$ as the usual operators, but acting just with respect to $x^h$.
In addition, we define the operator $\nabla^\perp_h\,:=\,\bigl(-\d_2\,,\,\d_1\bigr)$.
Finally, the symbol $\h$ denotes the Helmholtz projector onto the space of solenoidal vector fields in $B$, 
while $\h_h$ denotes the Helmholtz projection on $\R^2$.
Observe that, in the sense of Fourier multipliers, one has $\h_h\vec f\,=\,-\nabla_h^\perp(-\Delta_h)^{-1}\curlh\vec f$.

Moreover, since we will deal with a periodic problem in the $x^{3}$-variable, we also introduce the following decomposition: for a vector-field $X$, we write
\begin{equation} \label{eq:decoscil}
X(x)=\langle X\rangle (x^{h})+\widetilde{X}(x)\quad\qquad
 \text{ with }\quad \langle X\rangle(x^{h})\,:=\,\frac{1}{\left|\T\right|}\int_{\T}X(x^{h},x^{3})\, dx^{3}\,,
\end{equation}
where $\mbb{T}^1\,:=\,[-1,1]/\sim$ is the one-dimensional flat torus (here $\sim$ denotes the equivalence relation which identifies $-1$ and $1$)
and $\left|\T\right|$ denotes its Lebesgue measure.
Notice that $\widetilde{X}$ has zero vertical average, and therefore we can write $\widetilde{X}(x)=\d_{3}\widetilde{Z}(x)$ with $\widetilde{Z}$ having zero vertical average as well.

\subsection*{Acknowledgements}
{\small 
The work of the second and third authors has been partially supported by the project CRISIS (ANR-20-CE40-0020-01), operated by the French National Research Agency (ANR). The last author is supported by (Polish) National Center of Science grant 2020/38/E/ST1/00469.

The first and the third authors are members of the Italian Institute for Advanced Mathematics (INdAM) group. 
}

\section{Setting of the problem and main results} \label{s:result}

In this section, we introduce the primitive system and formulate our working hypotheses (see Section \ref{ss:FormProb}), then we state our main results
(in Section \ref{ss:results}).

 \subsection{The primitive system} \label{ss:FormProb}

As already said in the introduction, in this paper we assumed that the motion of the fluid is described by a rescaled version of the
barotropic Navier-Stokes system with Coriolis and gravitational forces.

Thus, given a small parameter $\veps\in\,]0,1]$, the system reads as follows:
\begin{align}
&	\partial_t \vre + \div (\vre\ue)=0 \label{ceq}\tag{NSC$_{\ep}^1$} \\
&	\partial_t (\vre\ue)+ \div(\vre\ue\otimes\ue) + \frac{1}{\ep}\,\e_3 \times \vre\ue +    \frac{1}{\ep^{2m}} \nabla_x p(\vre) 
	=\div \mbb{S}(\nabla_x\ue)  + \frac{\vre}{\ep^{2n}} \nabla_x G\, ,
	\label{meq}\tag{NSC$_{\ep}^2$} 
	\end{align}
where we recall that $m$ and $n$ are taken according to \eqref{eq:scale-our}.
The unknowns in the previous equations are the density $\vre=\vre(t,x)\geq0$ of the fluid and its velocity field $\ue=\ue(t,x)\in\R^3$, where $t\in\R_+$ and $x\in \Omega:=\R^2 \times\; ]0,1[$.
The viscous stress tensor in \eqref{meq} is given by Newton's rheological law
	\begin{equation}\label{S}
	\mbb{S}(\nabla_x \ue) = \mu\left( \nabla_x\ue + \nabla_x^T \ue  - \frac{2}{3}\div \ue \tens{Id} \right)
	+ \eta\, \div\ue \tens{Id}\,,
	\end{equation}
where $\mu>0$ is the shear viscosity and $\eta\geq 0$ represents the bulk viscosity. The term $\e_3\times\vrho_\veps\vec u_\veps$ takes into account the (strong) Coriolis force acting on the fluid.
As for the gravitational force, it is physically relevant to assume that 
	\begin{equation}\label{assG}
	 G(x)= -x^3\,.
	\end{equation}
The precise expression of $G$ will be useful in some computations below, 
although some generalisations are certainly possible.

	
The system is supplemented  with \emph{complete slip boundary conditions}, namely 
	\begin{align}
	\big(\ue \cdot \n\big) _{|\partial \Omega} = 0
	\quad &\mbox{ and } \quad
	\bigl([ \mbb{S} (\nabla_x \ue) \n ] \times \n\bigr)_{|\d\Omega} = 0\,,  \label{bc1-2}
	\end{align}
where $\vec{n}$ denotes the outer normal to the boundary $\d\Omega\,=\,\{x_3=0\}\cup\{x_3=1\}$.
Notice that this is a true simplification, because it avoid complications due to the presence of Ekman boundary layers, when passing to the limit
$\veps\ra0^+$.

\begin{remark} \label{r:period-bc}
As is well-known (see \tsl{e.g.} \cite{Ebin}), equations \eqref{ceq}--\eqref{meq}, supplemented by the complete slip boundary boundary conditions
from \eqref{bc1-2},
can be recasted as a periodic problem with respect to the vertical variable, in the new domain
$$
\Omega\,=\,\R^2\,\times\,\mbb{T}^1\,,\qquad\qquad\mbox{ with }\qquad\mbb{T}^1\,:=\,[-1,1]/\sim\,,
$$
where $\sim$ denotes the equivalence relation which identifies $-1$ and $1$. Indeed, the equations are invariant if we extend
$\rho$ and $u^h$ as even functions with respect to $x^3$, and $u^3$ as an odd function.

In what follows, we will always assume that such modifications have been performed on the initial data, and
that the respective solutions keep the same symmetry properties.
\end{remark}

Now we need to  impose structural restrictions on the pressure function $p$. We assume that
	\begin{equation}\label{pp1}
	p\in C^1 [0,\infty)\cap C^2(0,\infty),\qquad p(0)=0,\qquad p'(\vrho )>0\quad \mbox{ for all }\,\vrho\geq 0\, .
	\end{equation}
Additionally to \eqref{pp1}, we require that 
	\begin{equation}\label{pp2}
\mbox{ exists }\;\g\,>\,\frac{3}{2}\quad\mbox{ such that }\qquad
\lim\limits_{\vrho \to +\infty} \frac{p^\prime(\vrho)}{\vrho^{\gamma -1}} = p_\infty >0\, .
	\end{equation}
Without loss of generality, we can suppose that $p$ has been renormalised so that $p^\prime (1)=1$.  


\subsubsection{Equilibrium states} \label{sss:equilibrium}

Next, we focus our attention on the so-called \emph{equilibrium states}. For each value of $\veps\in\,]0,1]$ fixed, the equilibria of system \eqref{ceq}--\eqref{meq} consist of static densities $\vret$ satisfying
	\begin{equation}\label{prF}
\nabla_x p(\vret) = \ep^{2(m-n)} \vret \nabla_x G  \qquad \mbox{ in }\; \Omega\,.
	\end{equation}	

%
%

Equation \eqref{prF} identifies $\wtilde{\vrho}_\veps$ up to an additive constant: taking the target density to be $1$, we get
\begin{equation} \label{eq:target-rho}
  H^\prime(\vret)=\, \ep^{2(m-n)} G + H^\prime (1)\,,\qquad\qquad \mbox{ where }\qquad 
H(\vrho) = \vrho \int_1^{\vrho} \frac{ p(z)}{z^2} {\rm d}z\,.
\end{equation}
Notice that relation \eqref{eq:target-rho} implies that 
\begin{equation*}
H^{\prime \prime}(\vrho)=\frac{p^\prime (\vrho)}{\vrho} \quad \text{ and }\quad H^{\prime \prime}(1)=1\, .
\end{equation*}

Therefore, we infer that, whenever $m\geq1$ and $m>n$ as in the present paper, for any $x\in\Omega$ one has $\wtilde{\vrho}_\veps(x)\longrightarrow 1$ in the limit $\veps\ra0^+$.
More precisely, the next statement collects all the necessary properties of the static states. It corresponds to Lemma 2.3 and Proposition 2.5 of
\cite{DS-F-S-WK}. 
\begin{proposition} \label{p:target-rho_bound}
Let the  gravitational force $G$ be given by \eqref{assG}.
Let $\bigl(\wtilde{\vrho}_\veps\bigr)_{0<\veps\leq1}$ be a family of static solutions to equation \eqref{prF}
in $\Omega = \R^2\times\, ]0,1[$.

Then, there exist an $\veps_0>0$ and a $0<\rho_*<1$ such that $\wtilde{\vrho}_\veps\geq\rho_*$ for all $\veps\in\,]0,\veps_0]$
and all $x\in\Omega$.
In addition, for any $\veps\in\,]0,\veps_0]$, one has:
\begin{equation*}
\left|\wtilde{\vrho}_\veps(x)\,-\,1\right|\,\leq\,C\,\veps^{2(m-n)}\, ,
\end{equation*}
for a constant $C>0$ which is uniform in $x\in\Omega$ and in $\veps\in\,]0,1]$.
\end{proposition}

Without loss of any generality, we can assume that $\veps_0=1$ in Proposition \ref{p:target-rho_bound}.

\medbreak
In light of this analysis, it is natural to try to solve system \eqref{ceq}--\eqref{meq} in $\Omega$, supplemented with the \emph{far field conditions}
\begin{equation} \label{ff}
\varrho_{\varepsilon}\rightarrow \vret \qquad \mbox{ and } \qquad \ue \rightarrow 0 \qquad\qquad \text{ as }\quad |x|\rightarrow +\infty \, .
\end{equation}


\subsubsection{Initial data and finite energy weak solutions} \label{sss:data-weak}

In view of the boundary conditions \eqref{ff} ``at infinity'', we assume that the initial data are close (in a suitable sense) to the equilibrium states $\vret$ that we have just identified.
Namely, we consider initial densities of the following form:
	\begin{equation}\label{in_vr}
	\vrez = \vret + \ep^m \vrez^{(1)} \, .
	\end{equation}
For later use, let us introduce also the following decomposition of the initial densities: 
\begin{equation} \label{eq:in-dens_dec}
\vrho_{0,\veps}\,=\,1\,+\,\veps^{2(m-n)}\,R_{0,\veps}\qquad\qquad\mbox{ with }\qquad
R_{0,\veps}\,=\,\wtilde r_\veps\,+\,\veps^{2n-m}\, \vrho_{0,\veps}^{(1)}\,,\qquad \wtilde r_\veps\,:=\,\frac{\wtilde\vrho_\veps-1}{\veps^{2(m-n)}}\,.
\end{equation}
Notice that the $\wtilde r_\veps$'s are in fact data of the system, since they only depend on $p$ and $G$.

We suppose the density perturbations $\vrez^{(1)}$ to be measurable functions and satisfy the control
	\begin{align}
\sup_{\veps\in\,]0,1]}\left\|  \vrez^{(1)} \right\|_{(L^2\cap L^\infty)(\Omega)}\,\leq \,c\,,\label{hyp:ill_data}
	\end{align}
together with the ``mean-free condition''
$$
\int_{\Omega}  \vrez^{(1)} \dx = 0\,.
$$
As for the initial velocity fields, we assume the following uniform bound:
\begin{equation} \label{hyp:ill-vel}
 	\sup_{\veps\in\,]0,1]}\left\| \sqrt{\wtilde\vrho_\veps} \vec{u}_{0,\ep} \right\|_{L^2(\Omega)}\,  \leq\, c\,.
\end{equation}

\begin{remark} \label{r:ill_data}
In view of Proposition \ref{p:target-rho_bound}, the condition in \eqref{hyp:ill-vel} immediately implies that
$$
\sup_{\veps\in\,]0,1]}\,\left\| \vec{u}_{0,\ep}  \right\|_{L^2(\Omega)}\,\leq\,c\,.
$$
\end{remark}


Thanks to the previous uniform estimates, up to extraction, we can identify the limit points
\begin{align} 
\vrho^{(1)}_0\,:=\,\lim_{\veps\ra0}\vrho^{(1)}_{0,\veps}\qquad\text{ weakly-$\ast$ in }&\qquad L^\infty (\Omega) \cap L^2(\Omega)\label{conv:in_data_vrho}\\
\vec{u}_0\,:=\,\lim_{\veps\ra0}\vec{u}_{0,\veps}\qquad \text{ weakly in }&\qquad L^2(\Omega)\label{conv:in_data_vel}\,.
\end{align}

\medbreak


At this point, let us specify better what we mean for \emph{finite energy weak solution} (see \cite{F-N} for details). 
\begin{definition} \label{d:weak}
Let $\Omega = \R^2 \times\,  ]0,1[\,$. Fix  $T>0$ and $\veps>0$. Let $(\vrho_{0,\veps}, \vec u_{0,\veps})$ be 
an initial datum satisfying \eqref{in_vr}--\eqref{hyp:ill-vel}. We say that the couple $(\vrho_\veps, \vec u_\veps)$ is a \emph{finite energy weak solution} of the system 
\eqref{ceq}--\eqref{meq} in $\,]0,T[\,\times \Omega$,
supplemented with the boundary conditions \eqref{bc1-2} and far field conditions \eqref{ff}, related to the initial datum $(\vrho_{0,\veps}, \vec u_{0,\veps})$, if the following conditions hold:
\begin{enumerate}[(i)]
\item the functions $\vrho_\veps$ and $\ue$ belong to the class
\begin{equation*}
\vrho_\veps\geq 0\,,\; \vrho_\veps - \widetilde{\vrho}_\veps\,\in L^\infty\big(0,T; L^2+L^\gamma (\Omega)\big)\,,\;
\ue \in L^2\big(0,T;H^1(\Omega)\big),\; \big(\ue \cdot \n\big) _{|\partial \Omega} = 0\, ;
\end{equation*}
\item the equations have to be satisfied in a distributional sense:
	\begin{equation}\label{weak-con}
	-\int_0^T\int_{\Omega} \left( \vre \partial_t \varphi  + \vre\ue \cdot \nabla_x \varphi \right) \dxdt = 
	\int_{\Omega} \vrez \varphi(0,\cdot) \dx
	\end{equation}
for any $\varphi\in C^\infty_c([0,T[\,\times \overline\Omega)$ and
	\begin{align}
	&\int_0^T\!\!\!\int_{\Omega}  
	\left( - \vre \ue \cdot \partial_t \vec\psi - \vre [\ue\otimes\ue]  : \nabla_x \vec\psi 
	+ \frac{1}{\ep} \, \e_3 \times (\vre \ue ) \cdot \vec\psi  - \frac{1}{\ep^{2m}} p(\vre) \div \vec\psi  \right) \dxdt \label{weak-mom} \\
	& =\int_0^T\!\!\!\int_{\Omega} 
	\left(- \mbb{S}(\nabla_x\vec u_\veps)  : \nabla_x \vec\psi +  \frac{1}{\ep^{2n}} \vre \nabla_x G\cdot \vec\psi \right) \dxdt 
	+ \int_{\Omega}\vrez \uez  \cdot \vec\psi (0,\cdot) \dx \nonumber
	\end{align}
for any test function $\vec\psi\in C^\infty_c([0,T[\,\times \overline\Omega; \R^3)$ such that $\big(\vec\psi \cdot \n \big)_{|\partial {\Omega}} = 0$;
\item the energy inequality holds for almost every $t\in (0,T)$:
\begin{align}
&\hspace{-0.7cm} \int_{\Omega}\frac{1}{2}\vre|\ue|^2(t) \dx\,+\,\frac{1}{\ep^{2m}}\int_{\Omega}\mc E\left(\vrho_\veps,\wtilde\vrho_\veps\right)(t) \dx
+  \int_0^t\int_{\Omega} \mbb S(\nabla_x \ue):\nabla_x \ue \, \dx {\rm d}\tau  \label{est:dissip} \\
&\qquad\qquad\qquad\qquad\qquad\qquad\qquad\qquad
\,\leq\,
\int_{\Omega}\frac{1}{2}\vrez|\uez|^2 \dx\,+\,
\frac{1}{\ep^{2m}}\int_{\Omega}\mc E\left(\vrho_{0,\veps},\wtilde\vrho_\veps\right) \dx\, ,
\nonumber
\end{align}
where the function
\begin{equation} \label{def:rel-entropy}
\mc E\left(\rho,\wtilde\vrho_\veps\right)\,:=\,H(\rho) - (\rho - \vret)\, H^\prime(\vret)
- H(\vret)
\end{equation}
 is the \emph{relative internal energy} of the fluid, with $H$ given by \eqref{eq:target-rho}.
\end{enumerate}
The solutions is \emph{global} if the previous conditions hold for all $T>0$.
\end{definition}
	
Under the assumptions fixed above, for any \emph{fixed} value of the parameter $\veps\in\,]0,1]$,
the existence of a global in time finite energy weak solution $(\vrho_\veps,\vec u_\veps)$ to system \eqref{ceq}--\eqref{meq}, related to the initial datum
$(\vrho_{0,\veps},\vec u_{0,\veps})$, in the sense of the previous definition, can be proved as in the classical case, see \tsl{e.g.} \cite{Lions_2}, \cite{Feireisl}. 
Notice that the mapping $t \mapsto (\vre\ue)(t,\cdot)$ is weakly continuous, and one has $(\vre)_{|t=0} = \vrez$ together with $(\vre\ue)_{|t=0}= \vrez\uez$. 

We remark also that, in view of \eqref{ceq}, the total mass is conserved in time, in the following sense: for almost every $t\in[0,+\infty[\,$,
one has
\begin{equation*} \label{eq:mass_conserv}
\int_{\Omega}\bigl(\vre(t)\,-\,\vret\bigr)\,\dx\,=\,0\,.
\end{equation*}

To conclude, we point out that, in our framework of finite energy weak solutions to the primitive system,
inequality \eqref{est:dissip} will be the only tool to derive uniform estimates for the family of weak solutions we are going to consider.

\subsection{Main results}\label{ss:results}

We can now state our main results. We point out that, due to the scaling \eqref{eq:scale-our}, the relation $m>n$ is always true,
so we will always be in a low stratification regime.

The first statement concerns the case when the effects linked to the pressure term are predominant with respect to the fast rotation, \tsl{i.e.} $m>1$.

\begin{theorem}\label{th:m>1}
Let $\Omega= \R^2 \times\,]0,1[\,$ and $G\in W^{1,\infty}(\Omega)$ be as in \eqref{assG}. Take $m>1$ and $m+1\geq 2n >m$.
For any fixed value of $\veps \in \; ]0,1]$, let initial data $\left(\vrho_{0,\veps},\vec u_{0,\veps}\right)$ verify the hypotheses fixed in Paragraph \ref{sss:data-weak}, and let
$\left( \vre, \ue\right)$ be a corresponding weak solution to system \eqref{ceq}--\eqref{meq}, supplemented with the structural hypotheses  \eqref{S} on $\mbb{S}(\nabla_x \ue)$ and with boundary conditions \eqref{bc1-2} and far field conditions \eqref{ff}.
Let $\vec u_0$ be defined as in \eqref{conv:in_data_vel}.

Then, for any $T>0$ one has the convergence properties
	\begin{align*}
	\varrho_\ep \rightarrow 1 \qquad\qquad &\mbox{ strongly in } \qquad L^{\infty}\big(0,T; L_{\rm loc}^{\min\{2,\g\}}(\Omega )\big) \\
	\vec{u}_\ep \weak \vec{U}
	\qquad\qquad &\mbox{ weakly in }\qquad L^2\big(0,T;H^{1}(\Omega)\big)\,, 
	\end{align*}	
where $\vec{U} = (\vec U^h,0)$, with $\vec U^h=\vec U^h(t,x^h)$ such that $\divh\vec U^h=0$. In addition, the vector field $\vec{U}^h $ is a weak solution
to the following homogeneous incompressible Navier-Stokes system  in $\R_+ \times \R^2$,
\begin{align}
& \d_t \vec U^{h}+\divh\left(\vec{U}^{h}\otimes\vec{U}^{h}\right)+\nabla_h\Gamma-\mu \Delta_{h}\vec{U}^{h}=0\, , \label{eq_lim_m:momentum} 
\end{align}
for a suitable pressure function $\Gamma\in\mc D'(\R_+\times\R^2)$ and related to the initial condition
$$
\vec{U}_{|t=0}=\h_h\left(\lan\vec{u}^h_{0}\ran\right)\, .
$$
\end{theorem}



When $m=1$, the Mach and Rossby numbers have the same order of magnitude, and they keep in balance in the whole asymptotic process,
realising in this way the so-called \emph{quasi-geostrophic balance} in the limit. 
The next statement is devoted to this case.
\begin{theorem} \label{th:m=1}
Let $\Omega = \R^2 \times\,]0,1[\,$ and let $G\in W^{1,\infty}(\Omega)$ be as in \eqref{assG}. Take $m=1$ and $1/2<n<1$. 
For any fixed value of $\veps \in \; ]0,1]$, let initial data $\left(\vrho_{0,\veps},\vec u_{0,\veps}\right)$ verify the hypotheses fixed in Paragraph \ref{sss:data-weak}, and let
$\left( \vre, \ue\right)$ be a corresponding weak solution to system \eqref{ceq}--\eqref{meq}, supplemented with the structural hypotheses  \eqref{S} on $\mbb{S}(\nabla_x \ue)$ and with boundary conditions \eqref{bc1-2} and far field conditions \eqref{ff}.
Let $\left(\vrho^{(1)}_0,\vec u_0\right)$ be defined as in \eqref{conv:in_data_vrho} and \eqref{conv:in_data_vel}.

Then, for any $T>0$ one has the following convergence properties:
	\begin{align*}
	\varrho_\ep \rightarrow 1 \qquad\qquad &\mbox{ strongly in } \qquad L^{\infty}\big(0,T; L_{\rm loc}^{\min\{2,\g\}}(\Omega )\big) \\
	\vrho^{(1)}_\veps:=\frac{\varrho_\ep - \widetilde{\vrho_\veps}}{\ep}  \weakstar \vrho^{(1)} \qquad\qquad &\mbox{ weakly-$*$ in }\qquad L^{\infty}\big(0,T; L^{2}+L^{\gamma}(\Omega )\big) \\
	\vec{u}_\ep \weak \vec{U}
	\qquad\qquad &\mbox{ weakly in }\qquad L^2\big(0,T;H^{1}(\Omega)\big)\,,
	\end{align*}	
where, as above, $\vec{U} = (\vec U^h,0)$, with $\vec U^h=\vec U^h(t,x^h)$ such that $\divh\vec U^h=0$. 
Moreover, one has the relation $\vec U_h=\nabla_h^\perp \vrho^{(1)}$, and $\vrho^{(1)}$ satisfies (in the weak sense) the  quasi-geostrophic equation
\begin{align}
& \d_{t}\left(\vrho^{(1)}-\Delta_{h}\vrho^{(1)}\right) -\nabla_{h}^{\perp}\vrho^{(1)}\cdot
\nabla_{h}\left( \Delta_{h}\vrho^{(1)}\right) +\mu 
\Delta_{h}^{2}\vrho^{(1)}\,=\,0\,, \label{eq_lim:QG}  
\end{align}
supplemented with the initial condition
$$
\left(\vrho^{(1)}-\Delta_{h}\vrho^{(1)}\right)_{|t=0}= \langle \vrho_0^{(1)}\rangle-\curlh\lan\vec u^h_{0}\ran\,.
$$
\end{theorem}

\section{Consequences of the energy inequality} \label{s:energy}

In Definition \ref{d:weak}, we have postulated that the family of weak solutions $\big(\vrho_\veps,\vu_\veps\big)_\veps$ considered in Theorems
\ref{th:m>1} and \ref{th:m=1} satisfies the energy inequality \eqref{est:dissip}. 
In this section we take advantage of that fact  to infer uniform bounds for $\big(\vrho_\veps,\vu_\veps\big)_\veps$, see Section \ref{ss:unif-est}.
Thanks to those bounds, we can extract weak-limit points of the sequence of solutions and deduce some properties those limit points
have to satisfy, see Section \ref{ss:ctl1}.

\subsection{Uniform bounds and weak limits}\label{ss:unif-est}

This section is devoted to establish uniform bounds on the sequence $\bigl(\vrho_\veps,\vec u_\veps\bigr)_\veps$.
This can be done as in the classical case (see \tsl{e.g.} \cite{F-N} for details), since
the Coriolis term does not contribute to the total energy balance of the system.
However, for the reader's convenience, let us present some details.

To begin with, let us introduce a partition of the space domain $\Omega$ into the so-called ``essential'' and ``residual'' sets.
For this, for $t>0$ and for all $\veps\in\,]0,1]$, we define the sets
$$
\Omega_\ess^\veps(t)\,:=\,\left\{x\in\Omega\;\big|\quad \vrho_\veps(t,x)\in\left[1/2\,\rho_*\,,\,2\right]\right\}\,,\qquad\Omega^\veps_\res(t)\,:=\,\Omega\setminus\Omega^\veps_\ess(t)\,,
$$
where the positive constant $\rho_*>0$ has been defined in Proposition \ref{p:target-rho_bound}.
Then, given a function $h$, we write
$$
h\,=\,\left[h\right]_\ess\,+\,\left[h\right]_\res\,,\qquad\qquad\mbox{ where }\qquad \left[h\right]_\ess\,:=\,h\,\mathds{1}_{\Omega_\ess^\veps(t)}\,.
$$
Here above, $\mathds{1}_A$ denotes the characteristic function of a set $A\subset\Omega$.

Next, we observe that
\[
\Big[\mc E\big(\rho(t,x),\wtilde\vrho_\veps(x)\big)\Big]_\ess\,\sim\,\left[\rho-\wtilde\vrho_\veps(x)\right]_\ess^2
\qquad\quad \mbox{ and }\qquad\quad
\Big[\mc E\big(\rho(t,x),\wtilde\vrho_\veps(x)\big)\Big]_\res\,\geq\,C\left(1\,+\,\big[\rho(t,x)\big]_\res^\g\right)\,,
\]
where $\vret$ is the static density state identified in Section \ref{sss:equilibrium}  and $\mc E$ is given by  
\eqref{def:rel-entropy}.
Here above, the multiplicative constants are all strictly positive and may depend on $\rho_*$ and we agree to write $A\sim B$ whenever there exists a ``universal'' constant $c>0$ such that $(1/c)\, B\leq A\leq c\, B$.

Thanks to the previous observations, we easily see that, under the assumptions fixed in Section \ref{s:result} on the initial data,
the right-hand side of \eqref{est:dissip} is \emph{uniformly bounded} for all $\veps\in\,]0,1]$. Specifically, we have
$$
\int_{\Omega} \frac{1}{2}\vrez|\uez|^2\,\dx + \frac{1}{\ep^{2m}}\int_{\Omega}\mc E\left(\vrho_{0,\veps},
\, \wtilde\vrho_\veps\right)\,\dx\,\leq\,C\,.
$$

Owing to the previous inequalities and the finite energy condition \eqref{est:dissip} on the family of weak solutions,
it is quite standard to derive, for any time $T>0$ fixed and any $\veps\in\,]0,1]$, the following estimates:
\begin{align}
	\sup_{t\in[0,T]} \| \sqrt{\vre}\ue\|_{L^2(\Omega;\, \R^3)}\, &\leq\,c \label{est:momentum} \\	
	\sup_{t\in[0,T]} \left\| \left[ \dfrac{\vre - \vret}{\ep^m}\right]_\ess (t) \right\|_{L^2(\Omega)}\,&\leq\, c \label{est:rho_ess} \\
	\sup_{t\in[0,T]} \int_{\Omega}	\bbbone_{\mc{M}^\veps_\res[t]} \,dx\,&\leq \, c\,\ep^{2m} \label{est:M_res-measure}\\
	\sup_{t\in [0,T]} \int_{\Omega} [ \vre]^{\gamma}_\res (t)\,\dx \,
	\,&\leq\,c\,\ep^{2m} \label{est:rho_res} \\
	\int_0^T \left\| \nabla_x \ue + \nabla_x^T \ue  - \frac{2}{3} \div \ue \tens{Id} \right\|^2_{L^2(\Omega ;\, \R^{3\times3})}\, \dt\,
	&\leq\, c\, . \label{est:Du} 
	\end{align}
We refer to \cite{F-N} (see also \cite{F-G-N}, \cite{F-G-GV-N}, \cite{F_2019} and \cite{DS-F-S-WK}) for the details of the computations.

Owing to \eqref{est:Du} and a generalisation of the Korn inequality (see \tsl{e.g.} Chapter 10 of \cite{F-N}), we gather that
$\big(\nabla\vu_\veps\big)_\veps\,\subset\,L^2_T(L^2)$. On the other hand, by arguing as in \cite{F-G-N}, we can use
\eqref{est:momentum}, \eqref{est:M_res-measure} and \eqref{est:rho_res} to deduce that also
$\big(\vu_\veps\big)_\veps\,\subset\,L^2_T(L^2)$. Putting those bounds together, we finally infer that
\begin{equation}\label{unif-bound-for-vel}
\int_0^T \left\|\ue  \right\|^2_{H^{1}(\Omega ;\, \R^{3})}\, \dt\,\leq \, c\, .
\end{equation} 
In particular, there exist $\vU\,\in\,L^2_{\rm loc}\big(\R_+;H^1(\Omega;\R^3)\big)$ such that, up to a suitable extraction (not relabelled here),
we have
\begin{equation} \label{conv:u}
\vu_\veps\,\rightharpoonup\,\vU\qquad\qquad \mbox{ in }\quad L^2_{\rm loc}\big(\R_+;H^1(\Omega;\R^3)\big)\,.
\end{equation}

Let us move further and consider the density functions. The previous estimates on the density tell us that we must find a finer decomposition for the densities. As a matter of fact, for any time $T>0$ fixed, we have
	\begin{equation}\label{rr1}
\| \vre - 1 \|_{L^\infty_T(L^2 + L^{\gamma} + L^\infty)}\,\leq\,c\,  \ep^{2(m-n)}\,.
	\end{equation}
In order to see \eqref{rr1} we write 
\begin{equation}\label{rel:density_1}
|\vrho_\veps-1|\,\leq\,|\vrho_\veps-\widetilde{\vrho}_\veps|+|\widetilde{\vrho}_\veps-1|\,.
\end{equation}	
From \eqref{est:rho_ess}, we infer that $\big[\vr_\veps\,-\,\wtilde{\vr}_\veps\big]_\ess$ is of order $O(\veps^m)$ in $L^\infty_T(L^2)$.
For the residual part of the same term, we can use \eqref{est:rho_res} to discover that it is of order $O(\veps^{2m/\gamma})$.
Observe that, if $1<\g<2$, the higher order is $O(\veps^m)$, whereas, in the case $\g\geq2$, by use of \eqref{est:rho_res} and \eqref{est:M_res-measure} again, it is easy to get
\begin{equation} \label{est:res_g>2}
\left\|\left[\vr_\veps\,-\,\wtilde{\vr}_\veps\right]_\res\right\|_{L^\infty_T(L^2)}^2\,\leq\,C\,\veps^{2m}\,.
\end{equation}
Finally, we apply Proposition \ref{p:target-rho_bound} to control the last term in the right-hand side of \eqref{rel:density_1}.
In the end, estimate \eqref{rr1} is proved.

This having been established, and keeping in mind the notation introduced in \eqref{in_vr} and \eqref{eq:in-dens_dec}, we can introduce the
density oscillation functions
\[ 
R_\veps\,:=\, \frac{\varrho_\ep -1}{\ep^{2(m-n)}}\, =\,\wtilde{r}_\veps\,+\,\veps^{2n-m}\,\vrho_\veps^{(1)}\,,
\] 
where we have defined
\begin{equation} \label{def_deltarho}
\vrho_\veps^{(1)}(t,x)\,:=\,\frac{\vre-\wtilde{\vrho}_\veps}{\ep^m}\qquad\mbox{ and }\qquad
\wtilde{r}_\veps(x)\,:=\,\frac{\wtilde{\vrho}_\veps-1}{\ep^{2(m-n)}}\,.
\end{equation}

Thanks again to \eqref{est:rho_ess}, \eqref{est:rho_res} and Proposition \ref{p:target-rho_bound}, we see that the above quantities verify the following uniform bounds, for any time $T>0$ fixed:
\begin{equation}\label{uni_varrho1}
\sup_{\veps\in\,]0,1]}\left\|\vrho_\veps^{(1)}\right\|_{L^\infty_T(L^2+L^{\gamma}({\Omega}))}\,\leq\, c \qquad\qquad\mbox{ and }\qquad\qquad
\sup_{\veps\in\,]0,1]}\left\| \wtilde{r}_\veps \right\|_{L^{\infty}(\Omega)}\,\leq\, c \,.
\end{equation}
In view of the previous properties, there exist $\vrho^{(1)}\in L^\infty_{\rm loc}(\R_+;L^2+L^{\gamma})$ and
$\wtilde{r}\in L^\infty$ such that (up to the extraction of a new suitable subsequence), for any $T>0$ we have
\begin{equation} \label{conv:rr}
\vrho_\veps^{(1)}\,\weakstar\,\vrho^{(1)}\quad\mbox{ weakly-$*$ in } L^\infty(0,T;L^2+L^{\gamma}(\Omega))\qquad 
\mbox{ and }\qquad \wtilde{r}_\veps\,\weakstar\,\wtilde{r} \quad\mbox{ weakly-$*$ in } L^\infty(\Omega).
\end{equation}
In particular, we get
	\begin{equation*} 
	R_\veps\, \weakstar\,\widetilde{r} \quad \mbox{ weakly-$*$ in } L^\infty\bigl(0,T; L^{\min\{\g,2\}}_{\rm loc}(\Omega)\bigr)\, . 
	\end{equation*}

\begin{remark} \label{r:g>2}
Observe that, owing to \eqref{est:res_g>2}, when $\g\geq2$ we get
\[
\sup_{\veps\in\,]0,1]}\left\|\vrho_\veps^{(1)}\right\|_{L^\infty_T(L^2)}\,\leq\, c\, .
\]
Therefore, in that case we actually have that $\vrho^{(1)}\,\in\,L^\infty_T(L^2)$ and that $\vrho_\veps^{(1)}\,\stackrel{*}{\rightharpoonup}\,\vrho^{(1)}$ in $L^\infty_T(L^2)$.

Analogously, when $\g\geq2$ we also get
\[
\| \vre - 1 \|_{L^\infty_T(L^2 + L^\infty)}\,\leq\,c\,  \ep^{2(m-n)}\, .
\]

\end{remark}

\subsection{Constraints on the limit}\label{ss:ctl1}

In this section, we establish some properties that the limit points of the family $\bigl(\vrho_\veps,\vec u_\veps \bigr)_\veps$, which have been
identified here above, have to satisfy.

We first need a preliminary result about the decomposition of the pressure function, which will be useful in the following computations.
\begin{lemma}\label{lem:manipulation_pressure}
Let $(m,n)\in\R^2$ verify the condition $m+1\,\geq\,2n\,>\,m\, \geq 1$. Let $p$ be the pressure term satisfying the structural hypotheses \eqref{pp1}--\eqref{pp2}. Then, for any $\veps\in\,]0,1]$, one has
 \begin{equation}\label{p_decomp_lemma}
\begin{split}
\frac{1}{\veps^{2m}}\,\nabla_x\Big(p(\vrho_\veps)\,-\,p(\wtilde{\vrho}_\veps)\Big)\,=\,
\frac{1}{\veps^m}\nabla_x \Big(p^\prime (1)\vrho_\veps^{(1)}\Big)\,+\, 
\frac{1}{\veps^{2n-m}}\,\nabla_x \Pi_\veps\, ,
\end{split}
 \end{equation}
 where the functions $\vr_\veps^{(1)}$ have been introduced in \eqref{def_deltarho} and, for all $T>0$, the family $\big(\Pi_\veps\big)_\veps$ verifies the uniform bound
\begin{equation}\label{unif-bound-Pi}
\left\|\Pi_\veps\right\|_{L^\infty_T(L^1+L^2+L^\g)}\,\leq\,C\,.
\end{equation}
When $\g\geq2$, one can dispense of the space $L^\g$ in the above control of $\big(\Pi_\veps\big)_\veps$.
\end{lemma}

\begin{proof}
We start by writing simple algebraic computations:
\begin{equation}\label{rel_p}
\begin{split}
\frac{1}{\veps^{2m}}\,\nabla_x\Big(p(\vrho_\veps)\,-\,p(\wtilde{\vrho}_\veps)\Big)\,&=\,\frac{1}{\veps^{2m}}\,\nabla_x\Big(p(\vrho_\veps)\,-\,p(\wtilde{\vrho}_\veps)-p^\prime(\widetilde{\vrho}_\veps)(\vrho_\veps -\widetilde{\vrho}_\veps ) \Big)\\
&\qquad\qquad+\,\frac{1}{\veps^{m}}\,\nabla_x\Big(\big(p^\prime(\widetilde{\vrho}_\veps)\,-\,p'(1)\big)\,\vr_\veps^{(1)}\Big)\,+\,
\frac{1}{\veps^m}\nabla_x \Big(p^\prime (1)\vrho_\veps^{(1)}\Big)\,.
\end{split}
\end{equation}

We start by analysing the first term on the right-hand side of \eqref{rel_p}. For the essential part, we can employ a Taylor expansion
to write
$$
\left[p(\vrho_\veps)-p(\widetilde{\vrho}_\veps)-p^\prime(\widetilde{\vrho}_\veps)(\vrho_\veps-\widetilde{\vrho}_\veps)\right]_\ess=\left[p^{\prime \prime}(z_\veps )(\vrho_\veps - \widetilde{\vrho}_\veps)^2\right]_\ess\, ,
$$
where $z_\veps$ is a suitable point between $\vrho_\veps$ and $\widetilde{\vrho}_\veps$. Thanks to the uniform bound \eqref{est:rho_ess}, we have that this term is of order $O(\veps^{2m})$ in $L^\infty_T(L^1)$, for any $T>0$ fixed. For the residual part, we can use \eqref{est:M_res-measure} and \eqref{est:rho_res}, together with the boundedness of the profiles $\wtilde\vr_\veps$ (keep in mind Proposition \ref{p:target-rho_bound}), to deduce that
\[
\left\|\left[p(\vrho_\veps)-p(\widetilde{\vrho}_\veps)-p^\prime(\widetilde{\vrho}_\veps)(\vrho_\veps-\widetilde{\vrho}_\veps)\right]_\res\right\|_{L^\infty_T(L^1)}\,\leq C\,\veps^{2m}\,.
\]
We refer to \tsl{e.g.} Lemma 4.1 of \cite{F_2019} for details.

In a similar way, a Taylor expansion for the second term on the right-hand side of \eqref{rel_p} gives
$$
\big( p^\prime(\widetilde{\vrho}_\veps)-p^\prime(1)\big)\vrho^{(1)}_\veps\,=\,p''(\eta_\veps )( \widetilde{\vrho}_\veps -1)\vrho^{(1)}_\veps
$$
where $\eta_\veps$ is a suitable point between $\widetilde{\vrho}_\veps$ and 1. Owing to Proposition \ref{p:target-rho_bound} again and to bound \eqref{uni_varrho1},
we infer that this term is of order $O(\veps^{2(m-n)})$ in $L^\infty_T(L^2+L^\g)$, for any time $T>0$ fixed. Then, defining 
$$ \Pi_\veps:=\frac{1}{\veps^{2(m-n)}}\left[\frac{p(\vrho_\veps )-p(\widetilde{\vrho}_\veps)}{\veps^m}-p^\prime (1)\vrho_\veps^{(1)}\right] $$
we have the control \eqref{unif-bound-Pi}.

The final statement concerning the case $\g\geq2$ easily follows from Remark \ref{r:g>2}.
This completes the proof of the lemma. 
\qed
\end{proof}

\begin{remark} \label{r:pressure}
Notice that the last term appearing in \eqref{p_decomp_lemma} is singular in $\veps$. This is in stark contrast with the situation considered in previous works, see
\tsl{e.g.} \cite{F-G-N}, \cite{F-G-GV-N}, \cite{F-N_CPDE} and \cite{F_2019}. However, its gradient structure will play a fundamental role in the computations below.
\end{remark}

This having been pointed out, we can now analyse the constraints on the weak-limit points $\big(\vr^{(1)},\vec U\big)$, identified in relations \eqref{conv:u} and \eqref{conv:rr} above. 

\subsubsection{The case of large values of the Mach number: $m>1$} \label{ss:constr_2}

We start by considering the case of anisotropic scaling, namely $m>1$ and $m+1\geq 2n>m$. Notice that, in particular, one has $m>n$.

\begin{proposition} \label{p:limitpoint}
Let $m>1$ and $m+1\geq 2n>m$ in \eqref{ceq}--\eqref{meq}.
Let $\left( \vre, \ue \right)_{\veps}$ be a family of weak solutions, related to initial data $\left(\vrho_{0,\veps},\vec u_{0,\veps}\right)_\veps$
verifying the hypotheses of Section \ref{sss:data-weak}. Let $(\vrho^{(1)}, \vec{U} )$ be a limit point of the sequence
$\left(\vrho_\veps^{(1)}, \ue\right)_{\veps}$, as identified in Section \ref{ss:unif-est}. Then
\begin{align}
&\vec{U}\,=\,\,\Big(\vec{U}^h\,,\,0\Big)\,,\qquad\qquad \mbox{ with }\qquad \vec{U}^h\,=\,\vec{U}^h(t,x^h)\quad \mbox{ and }\quad \div_{\!h}\,\vec{U}^h\,=\,0\,,  \label{eq:anis-lim_1} \\[1ex]
&\nabla_x \vrho^{(1)}\,=\, 0
\qquad\qquad\mbox{ in }\;\,\mc D^\prime(\R_+\times \Omega)\,. \label{eq:anis-lim_2} 
\end{align}
\end{proposition}

\begin{proof} First of all, let us consider the weak formulation of the mass equation \eqref{ceq}. Take a test function
$\varphi\in C_c^\infty\bigl(\R_+\times\Omega\bigr)$  and denote $[0,T]\times K\,:=\,{\rm supp} \, \varphi$. Then by \eqref{weak-con} we have
$$
-\int^T_0\int_K\bigl(\vrho_\veps-1\bigr)\,\d_t\varphi \dxdt\,-\,\int^T_0\int_K\vrho_\veps\,\vec{u}_\veps\,\cdot\,\nabla_{x}\varphi \dxdt\,=\,
\int_K\bigl(\vrho_{0,\veps}-1\bigr)\,\varphi(0,\,\cdot\,)\dx\,.
$$
We can easily pass to the limit in this equation, thanks to the strong convergence $\vrho_\veps\longrightarrow1$, provided by \eqref{rr1}, and the weak convergence of
$\vec{u}_\veps$ in $L_T^2\bigl(L^6_{\rm loc}\bigr)$, provided by \eqref{conv:u} and Sobolev embeddings. Notice that one always has $1/\g\,+\,1/6\,\leq\,1$. In this way,
we find
 
\begin{equation}\label{001_U}
-\,\int^T_0\int_K\vec{U}\,\cdot\,\nabla_{x}\varphi \dxdt\,=\,0
\end{equation}
for $\varphi$ taken as above. Since the choice of $\varphi$ is arbitrary, 
we obtain that
\begin{equation} \label{eq:div-free}
\div \U = 0 \qquad\qquad\mbox{ a.e. in }\; \,\R_+\times \Omega\,.
\end{equation}

Next, we test the momentum equation \eqref{meq} on $\veps^m\,\vec\phi$, for a smooth compactly supported $\vec\phi$.
Using the uniform bounds established in Section \ref{ss:unif-est}, it is easy to see that the term presenting the time derivative, the viscosity term and the convective
term all converges to $0$, in the limit $\veps\ra0^+$. Since $m>1$, also the Coriolis term vanishes when $\veps\ra0^+$. It remains us to consider
the pressure and gravity terms in the weak formulation \eqref{weak-mom} of the momentum equation:
using relation \eqref{prF}, we see that we can couple them to write
\begin{align}
\frac{1}{\veps^{2m}}\,\nabla_x p(\vrho_\veps)-\, \frac{1}{\veps^{2n}}\,\vrho_\veps\nabla_x G\,=\,\frac{1}{\veps^{2m}}\nabla_x\Big(p(\vrho_\veps)\,-\,p(\wtilde{\vrho}_\veps)\Big)-\,\veps^{m-2n}\vrho_\veps^{(1)}\nabla_x G\,. \label{eq:mom_rest_1}
\end{align}
By \eqref{uni_varrho1} and the fact that $m>n$, we readily see that the last term in the right-hand side of \eqref{eq:mom_rest_1} converges to $0$,
when tested against any smooth compactly supported $\veps^m\,\vec\phi$.
At this point, we use Lemma \ref{lem:manipulation_pressure} to treat the first term on the right-hand side of \eqref{eq:mom_rest_1}. 
So, taking $\vec\phi \in C^\infty_c([0,T[\, \times \Omega)$ (for some $T>0$), we
test the momentum equation against $\veps^m\,\vec\phi$: using \eqref{conv:rr}, in the limit $\veps\ra0^+$ we find that 
$$ \int_0^T \int_\Omega p'(1) \vr^{(1)} \div  \vec\phi \dxdt = 0\,.$$
Recalling that $p^\prime (1)=1$, the previous relation implies \eqref{eq:anis-lim_2} for $\vr^{(1)}$.
In particular, that relation implies that $\vrho^{(1)}(t,x)\,=\,c(t)$ for almost all $(t,x)\in\R_+\times\Omega$, for a suitable function $c=c(t)$ depending only on time.

Now, in order to see effects due to the fast rotation in the limit, we need to ``filter out'' the contribution coming from the low Mach number.
To this end, we test \eqref{meq} on $\veps\,\vec\phi$, where this time we take $\vec\phi\,=\,\curl\vec\psi$, for some smooth compactly supported $\vec\psi\,\in C^\infty_c\bigl([0,T[\,\times\Omega\bigr)$, with $T>0$.
Once again, by uniform bounds we infer that the $\d_t$ term, the convective term and the viscosity term all converge to $0$ when $\veps\ra0^+$.
As for the pressure and the gravitational force, we argue as in \eqref{eq:mom_rest_1}. Since the structure of $\vec\phi$ kills any gradient term, we are left with the convergence
of the integral
$$
\int^T_0\int_\Omega\veps^{m-2n+1}\vrho_\veps^{(1)}\nabla_x G\cdot\vec\phi\,\dx\,dt\,\longrightarrow\,
\delta_0(m-2n+1)\int^T_0\int_\Omega\vrho^{(1)}\nabla_x G\cdot\vec\phi\,\dx\,dt\,,
$$
where $\de_0(\zeta)\,=\,1$ if $\zeta=0$, $\de_0(\zeta)\,=\,0$ otherwise.
Finally, arguing as done for the mass equation, we see that the Coriolis term converges to the integral $\int^T_0\int_\Omega\e_3\times\vec{U}\cdot\vec\phi$.

Consider the case $m+1>2n$ for a while.
Passing to the limit for $\veps\ra0^+$, we find that $\mbb{H}\left(\e_3\times\vec{U}\right)\,=\,0$, which implies that
$\e_3\times\vec{U}\,=\,\nabla_x\Phi$, for some potential function $\Phi$. From this relation, one easily deduces that $\Phi=\Phi(t,x^h)$, \tsl{i.e.} $\Phi$ does not depend
on $x^3$, and that the same property is inherited by $\vec{U}^h\,=\,\bigl(U^1,U^2\bigr)$, \tsl{i.e.} one has $\vec{U}^h\,=\,\vec{U}^h(t,x^h)$. Furthermore, since $\vec U^h\,=\,-\,\nabla^\perp\Phi$, we get that $\div_{\!h}\,\vec{U}^h\,=\,0$.
At this point, we
combine this fact with \eqref{eq:div-free} to infer that $\d_3 U^3\,=\,0$; but, thanks to the boundary condition
\eqref{bc1-2}, we must have $\bigl(\vec{U}\cdot\vec{n}\bigr)_{|\d\Omega}\,=\,0$, which implies that $U^3$ has to vanish at the boundary of $\Omega$.
Thus, we finally deduce that $U^3\,\equiv\,0$, whence \eqref{eq:anis-lim_1} follows.

Now, let us focus on the case when $m+1=2n$. The previous computations show that, when $\veps\ra0^+$, we get
\begin{equation}\label{eq:streamfunction_1} 
\vec{e}_{3}\times \vec{U}+\vrho^{(1)}\nabla_x G\,=\,\nabla_x\Phi \qquad\qquad\mbox{ in }\; \mc D^\prime(\R_+\times \Omega)\,,
\end{equation}
for a new suitable function $\Phi$. However, owing to \eqref{eq:anis-lim_2}, we see that $\vrho^{(1)}\nabla_x G\,=\,\nabla_x\big(\vr^{(1)}\,G\big)$; hence, the previous relations
can be recasted as $\e_3\times\vec U\,=\,\nabla_x\wtilde\Phi$, for a new scalar function $\wtilde\Phi$. Therefore, the same analysis as above applies,
allowing us to gather \eqref{eq:anis-lim_1} also in the case $m+1=2n$.
\qed
\end{proof}

\subsubsection{The case $m=1$} \label{ss:constr_1}

Now we focus on the case $m=1$. In this case, the fast rotation and weak compressibility
effects are of the same order: this allows to reach the so-called \emph{quasi-geostrophic balance} in the limit.


\begin{proposition}  \label{p:limit_iso}
Take $m=1$ and $1/2<n<1$ in system \eqref{ceq}--\eqref{meq}.
Let $\left( \vre, \ue\right)_{\veps}$ be a family of weak solutions to \eqref{ceq}--\eqref{meq}, associated with initial data
$\left(\vrho_{0,\veps},\vec u_{0,\veps}\right)$ verifying the hypotheses fixed in Section \ref{sss:data-weak}.
Let $(\vrho^{(1)}, \vec{U} )$ be a limit point of the sequence $\left(\vrho^{(1)}_{\veps} , \ue\right)_{\veps}$, as identified in Section \ref{ss:unif-est}.
Then,
\begin{align}
\vr^{(1)}\,=\,\vr^{(1)}(t,x^h)\qquad\mbox{ and }\qquad\vec{U}\,=\,\,\Big(\vec{U}^h\,,\,0\Big)\,,\qquad \mbox{ with }\quad 
\vec{U}^h\,=\,\nabla^\perp_h \vrho^{(1)}
\;\mbox{ a.e. in }\;\R_+ \times \R^2\,.  \label{eq:for q}
\end{align}
In particular, one has $\vec U^h\,=\,\vec U^h(t,x^h)$ and $\div_{\!h}\vec U^h\,=\,0$.
\end{proposition}
\begin{proof} Arguing as in the proof of Proposition \ref{p:limitpoint}, it is easy to pass to the limit in the continuity equation. In particular, we obtain again relation \eqref{eq:div-free} for $\vec U$.

Only the analysis of the momentum equation changes a bit with respect to the previous case $m>1$. Now, since the most singular terms are of order $\veps^{-1}$ (keep in mind Lemma \ref{lem:manipulation_pressure}), we test
the weak formulation \eqref{weak-mom} of the momentum equation against $\veps\,\vec\phi$, where $\vec \phi$ is a smooth compactly supported function. Similarly to what done above, the uniform bounds of Section \ref{ss:unif-est} allow us to infer that the only quantity which does not vanish in the limit is the sum of the terms involving the Coriolis force, the pressure and the gravitational force: more precisely, using also Lemma \ref{lem:manipulation_pressure} and \eqref{prF}, we have
$$
\vec{e}_{3}\times \vrho_{\veps}\ue\,+\frac{\nabla_x \Big( p(\vrho_\veps)-p(\widetilde{\vrho}_\veps)\Big)}{\veps}\,-\,
\veps^{2(1-n)}\vrho_\veps^{(1)}\nabla_x G\,=\,\mc O(\veps)
$$
in the sense of $\mc D'(\R_+\times\Omega)$.
Following the same computations performed in the proof of Proposition \ref{p:limitpoint}, in the limit $\veps\ra0^+$ it is easy to get that
$$ 
\vec{e}_{3}\times \vec{U}+\nabla_x\left(p^\prime (1) \vrho^{(1)}\right)\,=\,0\qquad\qquad\mbox{ in }\; \mc D'\big(\R_+\times \Omega\big)\,.
$$ 
After recalling that $p^\prime (1)=1$, this equality can be equivalently written as 
$$ 
\vec{e}_{3}\times \vec{U}+\nabla_x \vrho^{(1)}\,=\,0 \qquad\qquad\mbox{ a.e. in }\; \R_+ \times \Omega\,.
$$ 
Notice that $\vec U$ is in fact in $L^2_{\rm loc}(\R_+;L^2)$, therefore so is $\nabla_x \vr^{(1)}$; hence the previous relation is in fact satisfied almost everywhere
in $\R_+\times\Omega$.

At this point, we can repeat the same argument used in the proof of Proposition \ref{p:limitpoint} to deduce \eqref{eq:for q}.
The proposition is thus proved.
\qed
\end{proof}

\section{Convergence in the case $m>1$}\label{s:proof}

In this section, we complete the proof of Theorem \ref{th:m>1}. Namely, we show convergence in the weak formulation of the primitive system, in the case when $m>1$ and $m+1\geq 2n>m$. 

In Proposition \ref{p:limitpoint}, we have already seen how passing to the limit in the mass equation.
However, problems arise when tackling the convergence in the momentum equation. Indeed, the analysis carried out so far
is not enough to identify the weak limit of the convective term $\vrho_\veps\,\vec u_\veps\otimes\vec u_\veps$, which is highly non-linear.
For proving that this term converges to the expected limit $\vec U\otimes\vec U$, the key point is to control the strong oscillations in time of the solutions,
generated by the singular terms in the momentum equation. For this, we will use a compensated compactness argument and exploit the algebraic structure of the wave system
underlying the primitive equations \eqref{ceq}--\eqref{meq}.

In Section \ref{ss:acoustic}, we start by giving a quite accurate description of those fast oscillations. 
Then, using that description, we are able, in Section \ref{ss:convergence}, to establish two fundamental properties: on the one hand, strong convergence of a suitable quantity related to the velocity fields; on the other hand, that the other terms which do not involve that quantity tend to vanish when $\veps\ra0^+$.
In turn, this allows us to complete, in Section \ref{ss:limit}, the proof of the convergence.

\subsection{Analysis of the acoustic waves} \label{ss:acoustic}

The goal of the present subsection is to describe the fast time oscillations of the solutions. First of all, we recast our equations into a wave system. Then, we establish uniform bounds for the quantities appearing in the wave system. Finally, we apply a regularisation
in space procedure for all the quantities, which is preparatory in view of the computations of Section \ref{ss:convergence}.

\subsubsection{Formulation of the acoustic equation} \label{sss:wave-eq}

We introduce the quantity
$$
\vec{V}_\veps\,:=\,\vrho_\veps\vec{u}_\veps\,.
$$
Then, straightforward computations show that we can recast the continuity equation in the form
\begin{equation} \label{eq:wave_mass}
\veps^m\,\d_t\vrho^{(1)}_\veps\,+\,\div\vec{V}_\veps\,=\,0\,,
\end{equation}
where $\vrho^{(1)}_\veps$ is defined in \eqref{def_deltarho}.
Next, thanks to Lemma \ref{lem:manipulation_pressure} and the static relation \eqref{prF}, we can derive the following form of the momentum equation:
\begin{align}
\veps^m\,\d_t\vec{V}_\veps\,+\,\veps^{m-1}\,\e_3\times \vec V_\veps\,+p^\prime(1)\,\nabla_x \vrho_{\veps}^{(1)}\,&=\,
\veps^{2(m-n)}\left(\vrho_\veps^{(1)}\nabla_x G\,-\,\nabla_x\Pi_\veps 
\right) \label{eq:wave_momentum} \\
&\qquad\qquad
+\,\veps^m\,\Big(\div\mbb{S}\!\left(\nabla_x\vec{u}_\veps\right)\,-\,\div\!\left(\vrho_\veps\vec{u}_\veps\otimes\vec{u}_\veps\right)
\Big)\,. \nonumber
\end{align}

Then, if we define
\begin{equation}\label{def_f-g}
\vec f_\veps :=\div\big(\mbb{S}\!\left(\nabla_x\vec{u}_\veps\right)\,-\,\vrho_\veps\vec{u}_\veps\otimes\vec{u}_\veps\big)\qquad \mbox{ and }\qquad
\vec g_\veps :=\vrho_\veps^{(1)}\nabla_x G\,-\,\nabla_x\Pi_\veps\,,
\end{equation}
recalling that we have normalised the pressure function so that $p^\prime (1)=1$, we can recast the primitive system \eqref{ceq}--\eqref{meq} in the following form:
\begin{equation} \label{eq:wave_syst}
\left\{\begin{array}{l}
       \veps^m\,\d_t \vrho^{(1)}_\veps\,+\,\div\vec{V}_\veps\,=\,0 \\[1ex]
       \veps^m\,\d_t\vec{V}_\veps\,+\,\nabla_x \vrho_\veps^{(1)}\,+\,\veps^{m-1}\,\e_3\times \vec V_\veps\,=\,\veps^m\,\vec f_\veps +\veps^{2(m-n)}\vec g_\veps\,.
       \end{array}
\right.
\end{equation}

We remark that system \eqref{eq:wave_syst} has to be read in the weak sense: for any $\varphi\in C_c^\infty\bigl([0,T[\,\times \oline\Omega\bigr)$, one has
$$
-\,\veps^m\,\int^T_0\int_{\Omega} \vrho^{(1)}_\veps\,\d_t\varphi\,-\,\int^T_0\int_{\Omega} \vec{V}_\veps\cdot\nabla_x\varphi\,=\,
\veps^{m}\int_{\Omega} \vrho^{(1)}_{0,\veps}\,\varphi(0)\,\,,
$$
and also, for any $\vec{\psi}\in C_c^\infty\bigl([0,T[\,\times \oline\Omega;\R^3\bigr)$ such that $(\vec \psi \cdot \vec n)_{|\partial \Omega}=0$, one has
\begin{align*}
&\hspace{-0.5cm}
-\,\veps^m\,\int^T_0\int_{\Omega}\vec{V}_\veps\cdot\d_t\vec{\psi}\,-\,\int^T_0\int_{\Omega} \vrho^{(1)}_\veps\,\div\vec{\psi}\,+\,\veps^{m-1}\int^T_0\int_{\Omega} \e_3\times\vec V_\veps\cdot\vec\psi \\
&\qquad\qquad\qquad\qquad\qquad
=\,\veps^{m}\int_{\Omega}\vrho_{0,\veps}\,\vec{u}_{0,\veps}\cdot\vec{\psi}(0)\,+\,\veps^m\,\int^T_0\int_{\Omega} \vec f_\veps \cdot\vec{\psi}+\,\veps^{2(m-n)}\,\int^T_0\int_{\Omega} \vec g_\veps \cdot\vec{\psi}\,.
\end{align*}

Here we use estimates of Section \ref{ss:unif-est} in order to show uniform bounds for the solutions and the data in the wave equation \eqref{eq:wave_syst}.
We start by dealing with the ``unknown'' $\vec V_\veps$. Splitting the term into essential and residual parts, one can obtain for all $T>0$: 
\begin{equation}\label{eq:V_bounds}
\|\vec V_\veps\|_{L^\infty_T(L^2+L^{2\gamma/(\gamma +1)})}\leq c\, .
\end{equation}
In the next lemma, we establish bounds for the source terms in the system of acoustic waves \eqref{eq:wave_syst}.

\begin{lemma} \label{l:source_bounds}
Write $\vec f_\veps\,=\,\div \wtilde{\vec f}_\veps$ and $\vec g_\veps\,=\,\vec g^1_\veps\,-\,\nabla_x\Pi_\veps$, where we have defined
$\vec g^1_\veps\,:=\,\vr_\veps^{(1)}\,\nabla_xG$ and the functions $\Pi_\veps$ have been introduced in Lemma \ref{lem:manipulation_pressure}.

For any $T>0$ fixed, one has the uniform embedding properties
\[
\big(\wtilde{\vec f}_\veps\big)_\veps\,\subset\,L^2_T(L^2+L^1)\qquad\mbox{ and }\qquad \big(\vec g^1_\veps\big)_\veps\,\subset\,L^2_T(L^2+L^\g)\,.
\]
In the case $\g\geq2$, we may dispense with the space $L^\g$ in the control of $\big(\vec g^1_\veps\big)_\veps$.

In particular, the sequences $\bigl(\vec f_\veps\bigr)_\veps$ and
$\bigl(\vec g_\veps\bigr)_\veps$, defined in system \eqref{eq:wave_syst}, are uniformly bounded in the space $L^{2}\big([0,T];H^{-s}(\Omega)\big)$, for all $s>5/2$.
\end{lemma}

\begin{proof}
From \eqref{est:momentum}, \eqref{est:Du} and \eqref{unif-bound-for-vel}, we immediately infer the uniform bound for the family $\big(\wtilde{\vec f}_\veps\big)_\veps$ in $L^2_T(L^1+L^2)$,
from which we deduce also the uniform boundedness of $\big(\vec f_\veps\big)_\veps$ in $L^2_T(H^{-s})$, for any $s>5/2$.

Next, for bounding $\big(\vec g^1_\veps\big)_\veps$ we simply use \eqref{uni_varrho1}, together with Remark \ref{r:g>2} when $\g\geq2$.
Keeping in mind the bounds established in Lemma \ref{lem:manipulation_pressure}, the uniform estimate for $\big(\vec g_\veps\big)_\veps$ follows.
\qed
\end{proof}

\subsubsection{Regularization and description of the oscillations}\label{sss:w-reg}



As already mentioned in Remark \ref{r:period-bc}, in order to apply the Littlewood-Paley theory, it is convenient to reformulate problem \eqref{ceq}--\eqref{meq} in the new domain (which we keep calling $\Omega$, with a little abuse of notation) 
$$ \Omega:=\R^2 \times \TT^1\, , \quad \text{with}\quad \TT^1:=[-1,1]/\sim\, .$$ 
In addition, to avoid the appearing of (irrelevant) multiplicative constants in the computations, we suppose that the torus $\TT^1$ has been renormalised so that its Lebesgue measure is equal to 1.

Now, for any $M\in\N$ we consider the low-frequency cut-off operator ${S}_{M}$ of a Littlewood-Paley decomposition, as introduced in equation \eqref{eq:S_j} of the Appendix. Then, we define 
\begin{equation}\label{def_reg_vrho-V}
\vrho^{(1)}_{\varepsilon ,M}={S}_{M}\vrho^{(1)}_{\veps}\qquad\qquad \text{ and }\qquad\qquad \vec{V}_{\veps ,M}={S}_{M}\vec{V}_{\veps}\, .
\end{equation} 

The previous regularised quantities satisfy the following properties.

\begin{proposition} \label{p:prop approx}
For any $T>0$, we have the following convergence properties, in the limit $M\rightarrow +\infty $:
\begin{equation}\label{eq:approx var}
\begin{split}
&\sup_{0<\veps\leq1}\, \left\|\vrho^{(1)}_{\varepsilon }-\vrm\right\|_{L^{\infty}([0,T];H^{-s})}\longrightarrow 0\qquad
\forall\,s>\max\left\{0,3\left(\frac{1}{\g}\,-\,\frac12\right)\right\}\\
&\sup_{0<\veps\leq1}\, \left\|\vec{V}_{\varepsilon }-\vec{V}_{\varepsilon ,M}\right\|_{L^{\infty}([0,T];H^{-s})}\longrightarrow 0\qquad
\forall\,s>\frac{3}{2\,\g}\,.
\end{split}
\end{equation}
Moreover, for any $M>0$, the couple $(\vrm,\vec V_{\veps ,M})$ satisfies the approximate wave equations
\begin{equation}\label{eq:approx wave}
\left\{\begin{array}{l}
       \veps^m\,\d_t \vrm \,+\,\,\div\vec{V}_{\veps ,M}\,=\,0 \\[1ex]
       \veps^m\,\d_t\vec{V}_{\veps ,M}\,+\veps^{m-1}\,e_{3}\times \vec{V}_{\veps ,M}+\,\nabla_x \vrm\,=\,\veps^m\,\vec f_{\veps ,M}\,+\veps^{2(m-n)} \vec g_{\veps,M} 
       \end{array}
\right.
\end{equation}
where $(\vec f_{\veps ,M})_{\veps}$ and $(\vec g_{\veps ,M})_{\veps}$ are families of smooth (in the space variables) functions satisfying, for any $s\geq0$, the uniform bounds
\begin{equation}\label{eq:approx force}
\sup_{0<\veps\leq1}\, \left\|\vec f_{\veps ,M}\right\|_{L^{2}([0,T];H^{s})}\,+\,\sup_{0<\veps\leq1}\,\left\|\vec g_{\veps ,M}\right\|_{L^{\infty}([0,T];H^{s})}\,\leq\, C(s,M)\,,
\end{equation}
where the constant $C(s,M)$ depends on the fixed values of $s\geq 0$ and $M>0$, but not on $\veps>0$.
\end{proposition}

\begin{proof}
Thanks to characterization \eqref{eq:LP-Sob} of $H^{s}$, properties \eqref{eq:approx var} are straightforward consequences of the uniform bounds establish in Section \ref{ss:unif-est}. For instance, let us consider the functions $\vr^{(1)}_\veps$: when $\g\geq2$, owing to Remark \ref{r:g>2} one has
$\big(\vr^{(1)}_\veps\big)_\veps\,\subset\,L^\infty_T(L^2)$, and then we use estimate \eqref{est:sobolev} from the Appendix.
When $1<\g<2$, instead, we first apply the dual Sobolev embedding to infer that $\big(\vr^{(1)}_\veps\big)_\veps\,\subset\,L^\infty_T(H^{-\s})$,
with $\s\,=\,\s(\g)\,=\,3\big(1/\g-1/2\big)$, and then we use \eqref{est:sobolev} again. The bounds for the momentum
$\big(\vec V_\veps\big)_\veps$ can be deduced by a similar argument, after observing that $2\g/(\g+1)<2$ always.

Next, applying the operator ${S}_{M}$ to \eqref{eq:wave_syst} immediately gives us system \eqref{eq:approx wave}, where we have set 
\begin{equation*}
\vec f_{\veps ,M}:={S}_{M}\vec f_\veps \qquad \text{ and }\qquad \vec g_{\veps ,M}:={S}_{M}\vec g_\veps\,.
\end{equation*}
Thanks to Lemma \ref{l:source_bounds} and \eqref{eq:LP-Sob}, it is easy to verify inequality \eqref{eq:approx force}. 
\qed
\end{proof}

\medbreak
We will need also the following important decomposition for the momentum vector fields $\vec V_{\veps,M}$ and their $\curl$.
\begin{proposition} \label{p:prop dec}
For any $M>0$ and any $\veps\in\,]0,1]$, the following decompositions hold true:
\begin{equation*}
\vec{V}_{\veps ,M}\,=\,
\veps^{2(m-n)}\vec{t}_{\veps ,M}^{1}+\vec{t}_{\veps ,M}^{2}\qquad\mbox{ and }\qquad
\curl_{x}\vec{V}_{\veps ,M}=\veps^{2(m-n)}\vec{T}_{\veps ,M}^{1}+\vec{T}_{\veps ,M}^{2}\,,
\end{equation*}
where, for any $T>0$ and $s\geq 0$, one has 
\begin{align*}
&\left\|\vec{t}_{\veps ,M}^{1}\right\|_{L^{2}([0,T];H^{s})}+\left\|\vec{T}_{\veps ,M}^{1}\right\|_{L^{2}([0,T];H^{s})}\leq C(s,M) \\
&\left\|\vec{t}_{\veps ,M}^{2}\right\|_{L^{2}([0,T];H^{1})}+\left\|\vec{T}_{\veps ,M}^{2}\right\|_{L^{2}\left([0,T];L^2\right)}\leq C\,,
\end{align*}
for suitable positive constants $C(s,M)$ and $C$, which are uniform with respect to $\veps\in\,]0,1]$.
\end{proposition}

\begin{proof}
We decompose $\vec{V}_{\veps ,M}\,=\,\veps^{2(m-n)}\vec t_{\veps,M}^{1}\,+\,\vec t_{\veps,M}^{2}$, where we define
\begin{equation} \label{eq:t-T}
\vec{t}_{\veps,M}^{1}\,:=\,{S}_{M}\left(\frac{\vrho_\veps -1}{\veps^{2(m-n)}}\, \vec{u}_{\veps}\right) \qquad\mbox{ and }\qquad
\vec{t}_{\veps,M}^{2}\,:=\,{S}_{M}\left(\vec{u}_{\veps}\right)\,.
\end{equation}
The decomposition of $\curl_x\vec V_{\veps,M}$ follows after setting $\vec T_{\veps,M}^j\,:=\,\curl_x\vec t_{\veps,M}^j$, for $j=1,2$.

We have to prove uniform bounds for all those terms, by using the estimates established in Section \ref{ss:unif-est} above.
First of all, we have that $\big(\vu_\veps\big)_\veps\,\subset\,L^2_T(H^1)$, for any $T>0$ fixed. Then, we immediately gather the sought bounds for the vector fields $\vec t_{\veps,M}^2$ and $\vec T_{\veps,M}^2$.

For the families of $\vec t_{\veps,M}^1$ and $\vec T_{\veps,M}^1$, instead, we have to use the bounds provided by \eqref{rr1} and (when $\g\geq2$)
Remark \ref{r:g>2}. In turn, we see that for any $T>0$:
\[
\left(\frac{\vrho_\veps -1}{\veps^{2(m-n)}}\, \vec{u}_{\veps}\right)\,\subset\,L^2_T(L^1+L^2+L^{6\g/(\g+6)})\,\hookrightarrow\,
L^2_T(H^{-\s})\,,
\]
for some $\s>0$ large enough. Therefore, the claimed bounds follow thanks to the regularising effect of the operators $S_M$. The proof of the proposition
is thus completed.
\qed
\end{proof}


\subsection{Convergence of the convective term} \label{ss:convergence}
In this subsection we show the convergence of the convective term. 
The first step is to reduce its analysis to the case of smooth vector fields $\vec{V}_{\veps ,M}$.

\begin{lemma} \label{lem:convterm}
Let $T>0$. For any $\vec{\psi}\in C_c^\infty\bigl([0,T[\,\times\Omega;\R^3\bigr)$, we have 
\begin{equation*}
\lim_{M\rightarrow +\infty} \limsup_{\veps \rightarrow 0^+}\left|\int_{0}^{T}\int_{\Omega} \vrho_\veps\,\vec{u}_\veps\otimes \vec{u}_\veps: \nabla_{x}\vec{\psi}\, dx \, dt-
\int_{0}^{T}\int_{\Omega} \vec{V}_{\veps ,M}\otimes \vec{V}_{\veps,M}: \nabla_{x}\vec{\psi}\, dx \, dt\right|=0\, .
\end{equation*}
\end{lemma}

\begin{proof}
The proof is very similar to the one of Lemma 4.5 from \cite{DS-F-S-WK}, for this reason we just outline it.

One starts by using the decomposition $\vr_\veps\,=\,1\,+\,\veps^{2(m-n)}\,R_\veps$ to reduce (owing to the uniform bounds of Section \ref{ss:unif-est}) the convective term to the ``homogeneous counterpart'': for any test function $\vec\psi\in C^\infty_c\big(\R_+\times\Omega;\R^3\big)$, one has
\[
\lim_{\veps \rightarrow 0^+}\left|\int_{0}^{T}\int_{\Omega} \vrho_\veps\,\vec{u}_\veps\otimes \vec{u}_\veps: \nabla_{x}\vec{\psi}\, dx \, dt-
\int_{0}^{T}\int_{\Omega}\vec{u}_\veps\otimes\vec{u}_\veps:\nabla_{x}\vec{\psi}\,\dx\,\dt\right|\,=\,0\,.
\]
Notice that, here, one has to use that $\g\geq 3/2$.

After that, we write
$\vu_\veps\,=\,S_M(\vu_\veps)\,+\,(\Id-S_M)\vu_\veps\,=\,\vec t^2_{\veps,M}\,+\,(\Id-S_M)\vu_\veps$.
Using Proposition \ref{p:prop dec} and the fact that $\left\|(\Id-{S}_{M})\,\vec{u}_\veps\right\|_{L_{T}^{2}(L^{2})}\,\leq C\,2^{-M}\, \|\nabla_x\vec u_\veps\|_{L^2_T(L^2)}\leq C\,2^{-M}$, which holds in view of estimate \eqref{est:sobolev} from the Appendix and the uniform bound \eqref{unif-bound-for-vel},
one can conclude.
\qed
\end{proof}

\medbreak

From now on, for notational convenience, we  generically denote by $\mc{R}_{\veps ,M}$ any remainder term, that is any term satisfying the property
\begin{equation} \label{eq:reminder}
\lim_{M\rightarrow +\infty}\limsup_{\veps \rightarrow 0^+}\left|\int_{0}^{T}\int_{\Omega}\mc{R}_{\veps ,M}\cdot \vec{\psi}\, dx \, dt\right|=0
\end{equation}
for all test functions $\vec{\psi}\in C_c^\infty\bigl([0,T[\,\times\Omega;\R^3\bigr)$ lying in the kernel of the singular perturbation operator,
namely (in view of Proposition \ref{p:limitpoint}) such that
\begin{equation} \label{eq:test-f}
\vec\psi\in C_c^\infty\big([0,T[\,\times\Omega;\R^3\big)\qquad\qquad \mbox{ such that }\qquad \div\vec\psi=0\quad\mbox{ and }\quad \d_3\vec\psi=0\,.
\end{equation}
Notice that, in order to pass to the limit in the weak formulation of the momentum equation and derive the limit system, it is enough to use test functions $\vec\psi$ as above.

Thus, for $\vec\psi$ as in \eqref{eq:test-f}, 
we have to pass to the limit in the term 
\begin{align*}
-\int_{0}^{T}\int_{\Omega} \vec{V}_{\veps ,M}\otimes \vec{V}_{\veps ,M}: \nabla_{x}\vec{\psi}\,&=\,\int_{0}^{T}\int_{\Omega} \div\left(\vec{V}_{\veps ,M}\otimes
\vec{V}_{\veps ,M}\right) \cdot \vec{\psi}\,.
\end{align*}
Notice that the integration by parts above is well-justified, since all the quantities inside the integrals are smooth with respect to the space variable. Owing to the structure of the test function, and
resorting to the notation introduced in \eqref{eq:decoscil}, we remark that we can write
$$
\int_{0}^{T}\int_{\Omega} \div\left(\vec{V}_{\veps ,M}\otimes \vec{V}_{\veps ,M}\right) \cdot \vec{\psi}\,=\,
\int_{0}^{T}\int_{\R^2} \left(\mc{T}_{\veps ,M}^{1}+\mc{T}_{\veps, M}^{2}\right)\cdot\vec{\psi}^h\,,
$$
where we have defined the terms
\begin{equation} \label{def:T1-2}
\mc T^1_{\veps,M}\,:=\, \divh\left(\langle \vec{V}_{\veps ,M}^{h}\rangle\otimes \langle \vec{V}_{\veps ,M}^{h}\rangle\right)\qquad \mbox{ and }\qquad
\mc T^2_{\veps,M}\,:=\, \divh\left(\langle \widetilde{\vec{V}}_{\veps ,M}^{h}\otimes \widetilde{\vec{V}}_{\veps ,M}^{h}\rangle \right)\,.
\end{equation}

In the next two paragraphs, we will deal with each one of those terms separately. We borrow most of the arguments from \cite{DS-F-S-WK}
(see also \cite{F-G-GV-N}, \cite{F_2019} for a similar approach). However, the special structure of the gravity force will play a key role here,
in order (loosely speaking) to compensate the stronger singularity due to our scaling $2n>m$.

\subsubsection{Convergence of the vertical averages}\label{sss:term1}
We start by dealing with $\mc T^1_{\veps,M}$. It is standard to write
\begin{align}
\mc{T}_{\veps ,M}^{1}\,&=\,\divh\left(\langle \vec{V}_{\veps ,M}^{h}\rangle\otimes \langle \vec{V}_{\veps ,M}^{h}\rangle\right)=
\divh\langle \vec{V}_{\veps ,M}^{h}\rangle\, \langle \vec{V}_{\veps ,M}^{h}\rangle+\langle \vec{V}_{\veps ,M}^{h}\rangle \cdot \nabla_{h}\langle \vec{V}_{\veps ,M}^{h}\rangle \label{eq:T1} \\
&=\divh\langle \vec{V}_{\veps ,M}^{h}\rangle\, \langle \vec{V}_{\veps ,M}^{h}\rangle+\frac{1}{2}\, \nabla_{h}\left(\left|\langle \vec{V}_{\veps ,M}^{h}\rangle\right|^{2}\right)+
\curlh\langle \vec{V}_{\veps ,M}^{h}\rangle\,\langle \vec{V}_{\veps ,M}^{h}\rangle^{\perp}\,. \nonumber
\end{align}
Notice that the second term is a perfect gradient, so it vanishes when tested against divergence-free test functions. Hence, we can treat it as
a remainder, in the sense of \eqref{eq:reminder}.

For the first term in the second line of \eqref{eq:T1}, instead, we take advantage of system \eqref{eq:approx wave}: averaging the first equation with respect to $x^{3}$ and multiplying it by $\langle \vec{V}^h_{\veps ,M}\rangle$, we arrive at
$$
\divh\langle \vec{V}_{\veps ,M}^{h}\rangle\,\langle \vec{V}_{\veps ,M}^{h}\rangle\,=\,-\veps^m\d_t\langle \vrm\rangle \langle \vec{V}_{\veps ,M}^{h}\rangle\,=\,
\veps^m\langle \vrm\rangle \d_t \langle \vec{V}_{\veps ,M}^{h}\rangle +\mc{R}_{\veps ,M}\,,
$$
in the sense of distributions.
We remark that the term presenting the total derivative in time is in fact a remainder, thanks to the factor $\veps^m$ in front of it.
Now, we use the horizontal part of \eqref{eq:approx wave}, where we  first take the vertical average and then multiply by $\langle \vrm\rangle$:
since $m>1$, we gather
\begin{align*}
&\veps^m\langle \vrm \rangle \d_t \langle \vec{V}_{\veps ,M}^{h}\rangle \\
&\qquad\quad=
- \langle \vrm\rangle \nabla_{h}\langle \vrm \rangle-
\veps^{m-1}\langle \vrm \rangle\langle \vec{V}_{\veps ,M}^{h}\rangle^{\perp} 
+\veps^{m}\langle \vrm \rangle \langle \vec f_{\veps ,M}^{h}\rangle+\veps^{2(m-n)}\langle \vrm \rangle \langle \vec g_{\veps ,M}^{h}\rangle\\
&\qquad\quad=-\veps^{m-1}\langle \vrm \rangle\langle \vec{V}_{\veps ,M}^{h}\rangle^{\perp}+\mc{R}_{\veps ,M}\, ,
\end{align*}
where we have repeatedly exploited the properties proved in Proposition \ref{p:prop approx} and we have included in the remainder term also the perfect gradient.
Inserting this relation into \eqref{eq:T1} yields
\begin{equation*}
\mc{T}_{\veps ,M}^{1}=  \left(\curlh\langle \vec{V}_{\veps ,M}^{h}\rangle\,-\,\veps^{m-1}\langle \vrm \rangle\right)   
\langle\vec{V}_{\veps,M}^{h}\rangle^{\perp}+\mc{R}_{\veps,M}\,.
\end{equation*}
Observe that the first term appearing in the right-hand side of the previous relation is bilinear. Thus, in order to pass to the limit in it, one needs some strong convergence property. As a matter of fact, in the next computations we will work on the regularised wave system \eqref{eq:approx wave} to show that the quantity
\[
\gamma_{\veps, M}:=\curlh\langle \vec{V}_{\veps ,M}^{h}\rangle\,-\,\veps^{m-1}\langle \vrm \rangle
\]
is \emph{compact} in some suitable space. In particular, as $m>1$, also $\curlh\langle \vec{V}_{\veps ,M}^{h}\rangle$ is compact.

In order to see this, we write the vertical average of the first equation in \eqref{eq:approx wave} as
\begin{equation*}
\veps^{2m-1}\,\d_t \langle \vrm \rangle\,+\,\veps^{m-1}\div_{h} \langle \vec{V}_{\veps ,M}^{h}\rangle\,=0\,.
\end{equation*}
Next, we take the vertical average of the horizontal components of the second equation in \eqref{eq:approx wave} and then apply $\curlh$: we obtain
\begin{equation*}
\veps^m\,\d_t\curlh\langle \vec{V}_{\veps ,M}^{h}\rangle\,+\veps^{m-1}\,\divh\langle \vec{V}_{\veps ,M}^{h}\rangle\, =\,\veps^m \curlh\langle\vec f_{\veps ,M}^{h}\rangle+\veps^{2(m-n)} \curlh\langle\vec g_{\veps ,M}^{h}\rangle\, .
\end{equation*}
At this point, we recall the definition \eqref{def_f-g} of $\vec g_\veps$, and we see that $\curlh\langle\vec g_{\veps ,M}^{h}\rangle\,\equiv\,0$.
This property is absolutely fundamental, since it allows to erase the last term in the previous relation, which otherwise would have represented an obstacle to get compactness
of the $\g_{\veps,M}$'s. Indeed, thanks to this observation, we can sum up the last two equations to get
\begin{equation} \label{eq:gamma}
\d_{t}\gamma_{\veps,M}\,=\,\curlh\langle \vec f_{\veps ,M}^{h}\rangle\, .
\end{equation}
Using estimate \eqref{eq:approx force} in Proposition \ref{p:prop approx}, we discover that, for any $M>0$ fixed, the family 
$\left(\d_{t}\,\gamma_{\veps,M}\right)_{\veps}$ is uniformly bounded (with respect to $\veps$) in \tsl{e.g.} the space $L_{T}^{2}(L^{2})$. 
On the other hand, we have that, again for any $M>0$ fixed,
the sequence $(\gamma_{\veps,M})_{\veps}$ is uniformly bounded (with respect to $\veps$) \tsl{e.g.} in the space $L_{T}^{2}(H^{1})$.
Since the embedding $H_{\rm loc}^{1}\hookrightarrow L_{\rm loc}^{2}$ is compact, the Aubin-Lions Lemma implies that, for any $M>0$ fixed, the family $(\gamma_{\veps,M})_{\veps}$ is compact
in $L_{T}^{2}(L_{\rm loc}^{2})$. Then, up to extraction of a suitable subsequence (not relabelled here), that family converges strongly to a tempered distribution $\gamma_M$ in the same space. 

Now, we have that $\gamma_{\veps ,M}$ converges strongly to $\gamma_M$ in $L_{T}^{2}(L_{\rm loc}^{2})$ and $\langle \vec{V}_{\veps ,M}^{h}\rangle$ converges weakly to
$\langle \vec{V}_{M}^{h}\rangle$ in $L_{T}^{2}(L_{\rm loc}^{2})$ (owing to Proposition \ref{p:prop dec}, for instance). Then, we deduce that
\begin{equation*}
\gamma_{\veps,M}\langle \vec{V}_{\veps ,M}^{h}\rangle^{\perp}\longrightarrow \gamma_M \langle \vec{V}_{M}^{h}\rangle^{\perp}\qquad \text{ in }\qquad \mc{D}^{\prime}\big(\R_+\times\R^2\big)\,.
\end{equation*}
Observe that, by definition of $\g_{\veps,M}$, we must have $\gamma_M=\curlh\langle \vec{V}_{M}^{h}\rangle$. On the other hand, owing to Proposition \ref{p:prop dec} and \eqref{eq:t-T}, we know that $\langle \vec{V}_{M}^{h}\rangle= \lan{S}_{M}\vec{U}^{h}\ran$.
Therefore,  in the end we have proved that, for $m>1$ and $m+1\geq 2n >m$,
one has the convergence (at any $M\in\N$ fixed, when $\veps\ra0^+$)
\begin{equation} \label{eq:limit_T1}
\int_{0}^{T}\int_{\R^2}\mc{T}_{\veps ,M}^{1}\cdot\vec{\psi}^h\,dx^h\,dt\,\longrightarrow\,
\int^T_0\int_{\R^2}\curlh\lan{S}_{M}\vec{U}^{h}\ran\; \lan{S}_{M}(\vec{U}^{h})^{\perp}\ran\cdot\vec\psi^h\,dx^h\,dt\, ,
\end{equation}
for any $T>0$ and for any test-function $\vec \psi$ as in \eqref{eq:test-f}.

\subsubsection{Vanishing of the oscillations}\label{sss:term2}

We now focus on the term $\mc{T}_{\veps ,M}^{2}$, defined in \eqref{def:T1-2}. Recall that $m>1$. In what follows, we consider separately the two cases  $m+1>2n$ and $m+1=2n$. As a matter of fact, in the case when $m+1=2n$, a bilinear term involving $\vec g_{\veps,M}$ has no power of $\veps$ in front of it, so it is not clear that it converges to $0$ and, in fact, it might persist in the limit, giving rise to an additional term in the target system. To overcome this issue and show that this actually does not happen, we deeply exploit the structure of the wave system to recover a quantitative smallness for that term (namely, in terms of positive powers of $\veps$).

\subsubsection*{The case $m+1>2n$}\label{sss:term2_osc}

Starting from the definition of $\mc T_{\veps,M}^2$, the same computations as above yield
\begin{align}
\mc{T}_{\veps ,M}^{2}\,
&=\,\langle \divh (\widetilde{\vec{V}}_{\veps ,M}^{h})\;\;\widetilde{\vec{V}}_{\veps ,M}^{h}\rangle+\frac{1}{2}\, \langle \nabla_{h}| \widetilde{\vec{V}}_{\veps ,M}^{h}|^{2} \rangle+
\langle \curlh\widetilde{\vec{V}}_{\veps ,M}^{h}\,\left( \widetilde{\vec{V}}_{\veps ,M}^{h}\right)^{\perp}\rangle\, . \label{eq:T2} 
\end{align}

Let us now introduce the quantities
$$
\widetilde{\Phi}_{\veps ,M}^{h}\,:=\,( \widetilde{\vec{V}}_{\veps ,M}^{h})^{\perp}-\d_{3}^{-1}\nabla_{h}^{\perp}\widetilde{\vec{V}}_{\veps ,M}^{3}\qquad\mbox{ and }\qquad
\widetilde{\omega}_{\veps ,M}^{3}\,:=\,\curlh \widetilde{\vec{V}}_{\veps ,M}^{h}\,.
$$
Then we can write
\begin{equation*}
\left( \curl \widetilde{\vec{V}}_{\veps ,M}\right)^{h}\,=\,\d_3 \widetilde{\Phi}_{\veps ,M}^{h}\qquad \text{ and }\qquad
\left( \curl \widetilde{\vec{V}}_{\veps ,M}\right)^{3}\,=\,\widetilde{\omega}_{\veps ,M}^{3}\,.
\end{equation*}
In addition, from the momentum equation in \eqref{eq:approx wave}, where we take the mean-free part and then the $\curl$, we deduce the equations
\begin{equation} \label{eq:eq momentum term2}
\begin{cases}
\veps^{m}\d_t\widetilde{\Phi}_{\veps ,M}^{h}-\veps^{m-1}\widetilde{\vec{V}}_{\veps ,M}^{h}=\veps^m\left(\d_{3}^{-1}\curl\widetilde{\vec f}_{\veps ,M} \right)^{h}+\veps^{2(m-n)}\left(\d_{3}^{-1}\curl\widetilde{\vec g}_{\veps ,M} \right)^{h}\\[1ex]
\veps^{m}\d_t\widetilde{\omega}_{\veps ,M}^{3}+\veps^{m-1}\divh\widetilde{\vec{V}}_{\veps ,M}^{h}=\veps^m\,\curlh\widetilde{\vec f}_{\veps ,M}^{h}\, .
\end{cases}
\end{equation}
Making use of the relations above, recalling the definitions in \eqref{def_f-g}, and thanks to Propositions \ref{p:prop approx} and \ref{p:prop dec}, we can write
\begin{equation}\label{rel_oscillations}
\begin{split}
\curlh\widetilde{\vec{V}}_{\veps ,M}^{h}\;\left(\widetilde{\vec{V}}_{\veps ,M}^{h}\right)^{\perp}&=\widetilde{\omega}_{\veps ,M}^{3}\left(\widetilde{\vec{V}}_{\veps ,M}^{h}\right)^{\perp}\,\\
&=\veps \d_t\!\left( \widetilde{\Phi}_{\veps ,M}^{h}\right)^{\perp}\widetilde{\omega}_{\veps ,M}^{3}-
\veps\widetilde{\omega}_{\veps ,M}^{3}\left(\left(\d_{3}^{-1}\curl\widetilde{\vec f}_{\veps ,M}\right)^{h}\right)^\perp\\
&\qquad\qquad\qquad\qquad\qquad\qquad
-\veps^{m+1-2n}\, \widetilde{\omega}_{\veps ,M}^{3}\left(\left(\d_{3}^{-1}\curl\widetilde{\vec g}_{\veps ,M}\right)^{h}\right)^\perp  \\
&=-\veps \left( \widetilde{\Phi}_{\veps ,M}^{h}\right)^{\perp}\d_t\widetilde{\omega}_{\veps ,M}^{3}+\mc{R}_{\veps ,M}=
\left( \widetilde{\Phi}_{\veps ,M}^{h}\right)^{\perp}\,\divh\widetilde{\vec{V}}_{\veps ,M}^{h}+\mc{R}_{\veps ,M}\,,
\end{split}
\end{equation}
where, again, the equalities hold in the sense of distributions.
We point out that, thanks to the scaling $m+1>2n$, we could include in the remainder also the last term appearing in the second equality, which was of order $O(\veps^{m+1-2n})$. 

Hence, putting the gradient term into $\mc R_{\veps,M}$, from \eqref{eq:T2} we arrive at 
\begin{equation}\label{002_T}
\begin{split}
\mc{T}_{\veps ,M}^{2}\,&=\,\langle \divh\widetilde{\vec{V}}_{\veps ,M}^{h}\,\left(\widetilde{\vec{V}}_{\veps ,M}^{h}+\left(\widetilde{\Phi}_{\veps ,M}^{h}\right)^{\perp}\right) \rangle+\mc{R}_{\veps ,M} \\
&=\,\langle \div \widetilde{\vec{V}}_{\veps ,M}\left(\widetilde{\vec{V}}_{\veps ,M}^{h}+\left(\widetilde{\Phi}_{\veps ,M}^{h}\right)^{\perp}\right) \rangle -
\langle \d_3 \widetilde{\vec{V}}_{\veps ,M}^{3}\left(\widetilde{\vec{V}}_{\veps ,M}^{h}+\left(\widetilde{\Phi}_{\veps ,M}^{h}\right)^{\perp}\right) \rangle+\mc{R}_{\veps ,M}\, .
\end{split}
\end{equation}

At this point, the computations mainly follow the same lines of \cite{F-G-GV-N} (see also \cite{F_2019}, \cite{DS-F-S-WK}).
First of all, we notice that (in the last line) the second term on the right-hand side is another remainder. Indeed, using the definition of the function $\widetilde{\Phi}_{\veps ,M}^{h}$ and the fact
that the test function $\vec\psi$ does not depend on $x^3$, one has
\begin{equation*}
\begin{split}
\d_3 \widetilde{\vec{V}}_{\veps ,M}^{3}\left(\widetilde{\vec{V}}_{\veps ,M}^{h}+\left(\widetilde{\Phi}_{\veps ,M}^{h}\right)^{\perp}\right)&=\d_3 \left(\widetilde{\vec{V}}_{\veps ,M}^{3}\left(\widetilde{\vec{V}}_{\veps ,M}^{h}+\left(\widetilde{\Phi}_{\veps ,M}^{h}\right)^{\perp}\right)\right) - \widetilde{\vec{V}}_{\veps ,M}^{3}\, \d_3\left(\widetilde{\vec{V}}_{\veps ,M}^{h}+\left(\widetilde{\Phi}_{\veps ,M}^{h}\right)^{\perp}\right)\\
&=\mc{R}_{\veps ,M}-\frac{1}{2}\nabla_{h}\left|\wtilde{\vec{V}}_{\veps ,M}^{3}\right|^{2}=\mc{R}_{\veps ,M}\, .
\end{split}
\end{equation*}
Next, in order to deal with the first term on the right-hand side of \eqref{002_T}, we use the first equation in \eqref{eq:approx wave} to obtain
(in the sense of distributions)
\begin{equation*}
\begin{split}
\div \widetilde{\vec{V}}_{\veps ,M}\left(\widetilde{\vec{V}}_{\veps ,M}^{h}+\left(\widetilde{\Phi}_{\veps ,M}^{h}\right)^{\perp}\right)&=-\veps^{m} \d_t \widetilde{\vrho}^{(1)}_{\veps ,M}\left(\widetilde{\vec{V}}_{\veps ,M}^{h}+\left(\widetilde{\Phi}_{\veps ,M}^{h}\right)^{\perp}\right)+\mc{R}_{\veps ,M}\\
&=\veps^{m} \widetilde{\vrho}_{\veps,M}^{(1)}\, \d_t\left(\widetilde{\vec{V}}_{\veps ,M}^{h}+\left(\widetilde{\Phi}_{\veps ,M}^{h}\right)^{\perp}\right)+\mc{R}_{\veps ,M}\,.
\end{split}
\end{equation*}
Now, equations \eqref{eq:approx wave} and \eqref{eq:eq momentum term2} immediately yield that
\begin{equation*}
\veps^{m}\widetilde{\vrho}^{(1)}_{\veps ,M}\, \d_t\left(\widetilde{\vec{V}}_{\veps ,M}^{h}+\left(\widetilde{\Phi}_{\veps ,M}^{h}\right)^{\perp}\right)=
\mc{R}_{\veps ,M}-\widetilde{\vrho}_{\veps ,M}^{(1)}\, \nabla_{h}\widetilde{\vrho}^{(1)}_{\veps ,M}=
\mc{R}_{\veps ,M}-\frac{1}{2}\nabla_{h}\left|\widetilde{\vrho}_{\veps ,M}^{(1)}\right|^{2}=\mc{R}_{\veps ,M}\,.
\end{equation*}

This relation finally implies that $\mc{T}_{\veps ,M}^{2}\,=\,\mc R_{\veps,M}$ is a remainder, in the sense of relation \eqref{eq:reminder}:
for any $T>0$ and any test-function $\vec \psi$ as in \eqref{eq:test-f},
one has the convergence
(at any $M\in\N$ fixed, when $\veps\ra0^+$)
\begin{equation} \label{eq:limit_T2}
\int_{0}^{T}\int_{\R^2}\mc{T}_{\veps ,M}^{2}\cdot\vec{\psi}^h\,dx^h\,dt\,\longrightarrow\,0\,.
\end{equation}

\subsubsection*{The case $m+1=2n$}\label{sss:term2_osc_bis}
In the case $m+1=2n$, most of the previous computations may be reproduced exactly in the same way.
The only (fundamental) change concerns relation \eqref{rel_oscillations}: since now $m+1-2n=0$, that equation now reads
\begin{equation}\label{rel_oscillations_bis}
\begin{split}
\curlh\widetilde{\vec{V}}_{\veps ,M}^{h}\;\left(\widetilde{\vec{V}}_{\veps ,M}^{h}\right)^{\perp}\,=\,\left( \widetilde{\Phi}_{\veps ,M}^{h}\right)^{\perp}\,\divh\widetilde{\vec{V}}_{\veps ,M}^{h}\,- \widetilde{\omega}_{\veps ,M}^{3}\left(\left(\d_{3}^{-1}\curl\widetilde{\vec g}_{\veps ,M}\right)^{h}\right)^\perp+\mc{R}_{\veps ,M}\,,
\end{split}
\end{equation}
and, repeating the same computations performed for $\mc T^2_{\veps, M}$ in the previous paragraph, we have
\begin{equation*}\label{T^2-bis}
\mc T^2_{\veps, M}= \mc R_{\veps, M}-\langle\widetilde{\omega}_{\veps ,M}^{3}\left(\left(\d_{3}^{-1}\curl\widetilde{\vec g}_{\veps ,M}\right)^{h}\right)^\perp \rangle\, .
\end{equation*}
Hence, the main difference with respect to the previous case is that, now, we have to take care of the term $\widetilde{\omega}_{\veps ,M}^{3}\left(\left(\d_{3}^{-1}\curl\widetilde{\vec g}_{\veps ,M}\right)^{h}\right)^\perp $, which is non-linear and of order $O(1)$, so it may potentially
give rise to oscillations which persist in the limit.

In order to show that this does not happen, we make use of definition \eqref{def_f-g} of $\vec g_{\veps}$ to compute
\begin{align*}
\left(\curl\widetilde{\vec g}_{\veps ,M}\right)^{h,\perp}\,&=\,\left(\curl \left(\widetilde{\vrho}_{\veps, M}^{(1)}\nabla_x G-\nabla_x \widetilde{\Pi}_{\veps,M}\right)\right)^{h,\perp} \\
&=\,
\begin{pmatrix}
-\d_2 \widetilde{\vrho}^{(1)}_{\veps, M} \\ 
\d_1 \widetilde{\vrho}^{(1)}_{\veps ,M} \\ 
0
\end{pmatrix}^{h,\perp}\,=\,-\,\nabla_h \widetilde{\vrho}_{\veps,M}^{(1)}\, .
\end{align*}

From this relation, in turn we get
\begin{equation}\label{T^2}
\mc T^2_{\veps, M}\,=\,\mc R_{\veps, M}\,+\,\langle \widetilde{\omega}_{\veps ,M}^{3}\,  \d_3^{-1}\nabla_h \widetilde{\vrho}_{\veps,M}^{(1)} \rangle \, .
\end{equation} 
Now, we have to employ the potential part of the momentum equation in \eqref{eq:approx wave}, which has not been used so far. Taking the oscillating
component of the solutions, we obtain 
\begin{equation*}
\nabla_h \widetilde{\vrho}_{\veps,M}^{(1)}\,=-\, \veps^m\,\d_t\widetilde{\vec{V}}^h_{\veps ,M}\,-\veps^{m-1} (\widetilde{\vec{V}}^h_{\veps ,M})^\perp+\veps^m\,\widetilde{\vec f}^h_{\veps ,M}\,+\veps^{2(m-n)} \widetilde{\vec g}^h_{\veps,M}= -\, \veps^m\,\d_t\widetilde{\vec{V}}^h_{\veps ,M}\,+ \mc R_{\veps,M}\,.
\end{equation*}
Inserting this relation into \eqref{T^2} and using \eqref{eq:eq momentum term2}, we finally gather
\begin{equation*}
\mc T^2_{\veps, M}=-\veps^m \langle \widetilde{\omega}_{\veps ,M}^{3}\, \d_t\d_3^{-1}\widetilde{\vec{V}}^h_{\veps ,M}  \rangle +\mc R_{\veps,M}=
\veps^m \langle \d_t \widetilde{\omega}_{\veps ,M}^{3}\, \d_3^{-1}\widetilde{\vec{V}}^h_{\veps ,M}  \rangle +\mc R_{\veps,M}=\mc R_{\veps,M}\, ,
\end{equation*}
because we have taken $m>1$.

This relation finally implies that, also in the case when $m+1=2n$, $\mc{T}_{\veps ,M}^{2}$ is a remainder: for any $T>0$ and any test-function $\vec \psi$ as in \eqref{eq:test-f}, one has the convergence \eqref{eq:limit_T2}.

\subsection{The limit system} \label{ss:limit} 
Thanks to the computations of the previous subsection, we can now pass to the limit in equation \eqref{weak-mom}. Recall that $m>1$ and $m+1\geq 2n >m$ here.

To begin with, we take a divergence-free test function $\vec\psi$ as in \eqref{eq:test-f}, specifically
\begin{equation} \label{eq:test-2}
\vec{\psi}=\big(\nabla_{h}^{\perp}\phi,0\big)\,,\qquad\qquad\mbox{ with }\qquad \phi\in C_c^\infty\big([0,T[\,\times\R^2\big)\,,\quad \phi=\phi(t,x^h)\,.
\end{equation}
We point out that, since all the integrals will be made on $\R^2$ (in view of the choice of the test functions in \eqref{eq:test-2} above), we can safely work on the domain $\Omega=\R^2\times \, ]0,1[\,$.
In addition, for $\vec\psi$ as in \eqref{eq:test-2}, all the gradient terms vanish identically, as well as all the contributions
due to the vertical component of the equation. In particular, we do not see any contribution of the pressure and gravity terms: equation \eqref{weak-mom} becomes
\begin{align}
\int_0^T\!\!\!\int_{\Omega}  
& \left( -\vre \ue^h \cdot \partial_t \vec\psi^h -\vre \ue^h\otimes\ue^h  : \nabla_h \vec\psi^h
+ \frac{1}{\ep}\vre\big(\ue^{h}\big)^\perp\cdot\vec\psi^h\right)\, dx \, dt \label{eq:weak_to_conv}\\
&\qquad\qquad\qquad\qquad =-\int_0^T\!\!\!\int_{\Omega} 
\mbb{S}(\nabla_x\vec\ue): \nabla_x \vec\psi\,dx\,dt+
\int_{\Omega}\vrez \uez  \cdot \vec\psi(0,\cdot)\,dx\,. \nonumber
\end{align}

Making use of the uniform bounds of Section \ref{ss:unif-est}, we can pass to the limit in the $\d_t$ term and in the viscosity term.
Moreover, our assumptions imply that $\vrho_{0,\veps}\vec{u}_{0,\veps}\rightharpoonup \vec{u}_0$ in e.g. $L_{\rm loc}^2$. 
Next, the Coriolis term can be dealt with in a standard way: using the structure of $\vec\psi$ and the mass equation \eqref{weak-con}, we can write
\begin{align*}
\int_0^T\!\!\!\int_{\Omega}\frac{1}{\ep}\vre\big(\ue^{h}\big)^\perp\cdot\vec\psi^h\,&=\,\int_0^T\!\!\!\int_{\mbb{R}^2}\frac{1}{\ep}\langle\vre \ue^{h}\rangle \cdot \nabla_{h}\phi\,=\,
-\veps^{m-1}\int_0^T\!\!\!\int_{\mbb{R}^2}\langle \vrho^{(1)}_\veps\rangle\, \d_t\phi\,-\,\veps^{m-1}\int_{\mbb{R}^2}\langle  \vrho^{(1)}_{0,\veps}\rangle\, \phi(0,\cdot )\,, 
\end{align*}
which of course converges to $0$ when $\veps\ra0^+$.

It remains us to tackle the convective term $\vrho_\veps \ue^h \otimes \ue^h$.
For it, we take  advantage of Lemma \ref{lem:convterm} and relations \eqref{eq:limit_T1} and \eqref{eq:limit_T2}, but
we still have to take care of the convergence for $M\ra+\infty$ in \eqref{eq:limit_T1}.
We start by performing equalities \eqref{eq:T1} backwards in the term on the right-hand side of \eqref{eq:limit_T1}: thus, we have to pass to the limit for $M\ra+\infty$
in
\[
\int^T_0\int_{\R^2}\vec U_M^h\otimes\vec U_M^h : \nabla_h \vec\psi^h\,\dx^h\,\dt\,.
\]
Now, we remark that, since $\vec U^h\in L^2_T(H^1)$ by \eqref{conv:u}, from \eqref{eq:LP-Sob} we gather the strong convergence
$S_M \vec U^h\longrightarrow \vec{U}^{h}$ in $L_{T}^{2}(H^{s})$ for any $s<1$, in the limit for $M\rightarrow +\infty$.
Then, passing to the limit for $M\ra+\infty$ in the previous relation is an easy task: we finally get that, for $\veps\ra0^+$, one has
\begin{equation*}
\int_0^T\int_{\Omega} \vre \ue^h\otimes\ue^h  : \nabla_h \vec\psi^h\, \longrightarrow\, \int_0^T\int_{\R^2}\vec{U}^h\otimes\vec{U}^h  : \nabla_h \vec\psi^h\,.
\end{equation*}

In the end, we have shown that, letting $\varepsilon \rightarrow 0^+$ in \eqref{eq:weak_to_conv}, one obtains
\begin{align*}
&\int_0^T\!\!\!\int_{\R^2} \left(\vec{U}^{h}\cdot \d_{t}\vec\psi^h+\vec{U}^{h}\otimes \vec{U}^{h}:\nabla_{h}\vec\psi^h\right)\, dx^h\, dt=
\int_0^T\!\!\!\int_{\R^2} \mu \nabla_{h}\vec{U}^{h}:\nabla_{h}\vec\psi^h \, dx^h\, dt-
\int_{\R^2}\lan\vec{u}_{0}^{h}\ran\cdot \vec\psi^h(0,\cdot)\,dx^h\, ,
\end{align*}
for any test function $\vec\psi$ as in \eqref{eq:test-f}.
This implies \eqref{eq_lim_m:momentum}, concluding the proof of Theorem \ref{th:m>1}.

\section{Proof of the convergence for $m=1$} \label{s:proof-1}

In the present section, we complete the proof of the convergence in the case $m=1$ and $1/2<n<1$.
We will use again the compensated compactness argument depicted in Section \ref{ss:convergence},
and in fact most of the computations apply also in this case.

\subsection{Analysis of the acoustic-Poincar\'e waves}\label{ss:unifbounds_1} 

When $m=1$, the wave system \eqref{eq:wave_syst} takes the form
\begin{equation} \label{eq:wave_m=1}
\left\{\begin{array}{l}
       \veps\,\d_t \vrho_\veps^{(1)}\,+\,\div\vec{V}_\veps\,=\,0 \\[1ex]
       \veps\,\d_t\vec{V}_\veps\,+\,\nabla_x \vrho^{(1)}_\veps\,+\,\,\e_3\times \vec V_\veps\,=\,\veps\,\vec f_\veps+\veps^{2(1-n)}\vec g_\veps\,,
       \end{array}
\right.
\end{equation}
where $\bigl(\vrho^{(1)}_\veps\bigr)_\veps$  and $\bigl(\vec V_\veps\bigr)_\veps$ are defined as in Section \ref{sss:wave-eq}.
This system is supplemented with the initial datum $\big(\vrho^{(1)}_{0,\veps},\vr_{0,\veps}\vec u_{0,\veps}\big)$.

Next, we regularise all the quantities, by applying the Littlewood-Paley cut-off operator $S_M$ to \eqref{eq:wave_m=1}: we deduce that $\vrho^{(1)}_{\veps,M}$ and $\vec V_{\veps,M}$, defined as in \eqref{def_reg_vrho-V}, satisfy the regularised wave system
\begin{equation} \label{eq:reg-wave}
\left\{\begin{array}{l}
       \veps\,\d_t \vrho_{\veps,M}^{(1)}\,+\,\div\vec{V}_{\veps,M}\,=\,0 \\[1ex]
       \veps\,\d_t\vec{V}_{\veps,M}\,+\,\nabla_x \vrho^{(1)}_{\veps,M}\,+\,\,\e_3\times \vec V_{\veps,M}\,=\,\veps\,\vec f_{\veps,M}+\veps^{2(1-n)}\vec g_{\veps,M}\, ,
       \end{array}
\right.
\end{equation}
in the domain $\R_+\times\Omega$, where we recall that $\vec f_{\veps,M}:=S_M \vec f_\veps$ and $\vec g_{\veps,M}:=S_M \vec g_\veps$.
It goes without saying that a result similar to Proposition \ref{p:prop approx} holds true also in this case.

As it is apparent from the wave system \eqref{eq:wave_m=1} and its regularised version, when $m=1$ the pressure term and the Coriolis term are in balance, since they are of the same order. This represents the main change with respect to the case $m>1$, and it comes into play in the compensated compactness argument. Therefore, despite most of the computations may be repeated identical as in the previous section, let us present
the main points of the argument.

\subsection{Handling the convective term when $m=1$} \label{ss:convergence_1}

Let us take care of the convergence of the convective term in the case when $m=1$. 

First of all, it is easy to see that Lemma \ref{lem:convterm} still holds true. Therefore, 
given a test function $\vec\psi\in C_c^\infty\big([0,T[\,\times\Omega;\R^3\big)$ such that $\div\vec\psi=0$ and $\d_3\vec\psi=0$,
we have to pass to the limit in the term
\begin{align*}
-\int_{0}^{T}\int_{\Omega} \vec{V}_{\veps ,M}\otimes \vec{V}_{\veps ,M}: \nabla_{x}\vec{\psi}\,&=\,
\int_{0}^{T}\int_{\Omega}\div\left(\vec{V}_{\veps ,M}\otimes \vec{V}_{\veps ,M}\right) \cdot \vec{\psi}\,=\,
\int_{0}^{T}\int_{\R^2} \left(\mc{T}_{\veps ,M}^{1}+\mc{T}_{\veps, M}^{2}\right)\cdot\vec{\psi}^h\,,
\end{align*}
where we agree again that the torus $\T$ has been normalised so that its Lebesgue measure is equal to $1$ and we have adopted the same notation as in \eqref{def:T1-2}.

At this point, we notice that the analysis of $\mc{T}_{\veps ,M}^{2}$ can be performed as in Section \ref{sss:term2}, because we have
$m+1>2n$, \tsl{i.e.} $n<1$. \tsl{Mutatis mutandis}, we find relation \eqref{eq:limit_T2} also in the case $m=1$.

Let us now deal with the term $\mc{T}_{\veps ,M}^{1}$. Arguing as in Section \ref{sss:term1}, we may write it as
\begin{equation*}
\mc{T}_{\veps ,M}^{1}\,=\,\left(\curlh\langle \vec{V}_{\veps ,M}^{h}\rangle-\langle \vrho^{(1)}_{\veps ,M}\rangle \right)\langle \vec{V}_{\veps ,M}^{h}\rangle^{\perp}+\mc{R}_{\veps ,M} .
\end{equation*}
Now we use the horizontal part of \eqref{eq:reg-wave}: 
averaging it with respect to the vertical variable and applying the operator $\curlh$, we find
\begin{equation*}
\veps\,\d_t\curlh\langle \vec{V}_{\veps ,M}^{h}\rangle\,+\,\divh\langle \vec{V}_{\veps ,M}^{h}\rangle \,=\,
\veps\, \curlh\langle \vec f_{\veps ,M}^{h}\rangle\, .
\end{equation*}
Taking the difference of this equation with the first one in \eqref{eq:reg-wave}, we discover that
\begin{equation*}
\d_t\wtilde\g_{\veps,M}
\,=\,\curlh\langle \vec f_{\veps ,M}^{h}\rangle\,,\qquad\qquad \mbox{ where }\qquad
\wtilde\gamma_{\veps, M}:=\curlh\langle \vec{V}_{\veps ,M}^{h}\rangle\,-\,\langle \vrho^{(1)}_{\veps ,M}\rangle\,.
\end{equation*}
An argument analogous to the one used after \eqref{eq:gamma} above, based on Aubin-Lions Lemma, shows that
$\big(\wtilde\gamma_{\veps,M}\big)_{\veps}$ is compact in \tsl{e.g.} $L_{T}^{2}(L_{\rm loc}^{2})$. Then, this sequence converges strongly (up to extraction of a suitable subsequence, not relabelled here) to a tempered distribution $\wtilde\gamma_M$ in the same space. 

Using the previous property, we may deduce that
\begin{equation*}
\wtilde\gamma_{\veps,M}\,\langle \vec{V}_{\veps ,M}^{h}\rangle^{\perp}\,\longrightarrow\, \wtilde\gamma_M\, \langle \vec{V}_{M}^{h}\rangle^{\perp}\qquad \text{ in }\qquad \mc{D}^{\prime}\big(\R_+\times\R^2\big),
\end{equation*}
where we have $\langle \vec{V}_{M}^{h}\rangle=\lan S_M\vec{U}^{h}\ran$ and $\wtilde\gamma_M=\curlh\lan S_M \vec{U}^{h}\ran-\langle \vrho^{(1)}_{M}\rangle$.

Owing to the regularity of the target velocity $\vec U^h$, we can pass to the limit also for $M\ra+\infty$, as detailed in Section \ref{ss:limit} above. Thus, we find
\begin{equation} \label{eq:limit_T1-1}
\int^T_0\!\!\!\int_{\Omega}\vrho_\veps\,\vec{u}_\veps\otimes \vec{u}_\veps: \nabla_{x}\vec{\psi}\, dx \, dt\,\longrightarrow\,
\int^T_0\!\!\!\int_{\R^2}\big(\vec U^h\otimes\vec U^h:\nabla_h\vec\psi^h\,-\, \vrho^{(1)}\,(\vec U^h)^\perp\cdot\vec\psi^h\big)\,dx^h\,dt,
\end{equation}
for all test functions $\vec\psi$ such that $\div\vec\psi=0$ and $\d_3\vec\psi=0$. Recall the convention $|\T|=1$.
Notice that, since $\vec U^h=\nabla_h^\perp \vrho^{(1)}$ when $m=1$ (keep in mind Proposition \ref{p:limit_iso}), the last term in the integral on the right-hand side is actually zero.

\subsection{End of the proof} \label{ss:limit_1}
Thanks to the previous analysis, we are now ready to pass to the limit in equation \eqref{weak-mom}.
For this, we take a test-function $\vec\psi$ as in \eqref{eq:test-2};
notice in particular that $\div\vec\psi=0$ and $\d_3\vec\psi=0$. Then, once again all the gradient terms and all the contributions coming from the vertical
component of the momentum equation vanish identically, when tested against such a $\vec\psi$. Recall that all the integrals will be performed in $\R^2$. So, equation \eqref{weak-mom} reduces
to
\begin{align*}
\int_0^T\!\!\!\int_{\Omega}  \left( -\vre \ue \cdot \partial_t \vec\psi -\vre \ue\otimes\ue  : \nabla \vec\psi
+ \frac{1}{\ep}\vre\big(\ue^{h}\big)^\perp\cdot\vec\psi^h+\mbb{S}(\nabla_x\vec\ue): \nabla_x \vec\psi\right)
 =\int_{\Omega}\vrez \uez  \cdot \vec\psi(0,\cdot)\,.
\end{align*}

For the rotation term, we can test the first equation in \eqref{eq:wave_m=1} against $\phi$ to get
\begin{equation*} 
\begin{split}
-\int_0^T\!\!\!\int_{\R^2} \left( \lan \vrho^{(1)}_{\varepsilon}\ran\, \d_{t}\phi +\frac{1}{\veps}\, \lan\vrho_{\veps}\ue^{h}\ran\cdot \nabla_{h}\phi\right)=
\int_{\R^2}\lan \vrho^{(1)}_{0,\varepsilon }\ran\, \phi (0,\cdot ) \, ,
\end{split}
\end{equation*}
whence we deduce that
\begin{align*}
\int_0^T\!\!\!\int_{\Omega}\frac{1}{\ep}\vre\big(\ue^{h}\big)^\perp\cdot\vec\psi^h\,&=\,\int_0^T\!\!\!\int_{\mbb{R}^2}\frac{1}{\ep}\langle\vre \ue^{h}\rangle \cdot \nabla_{h}\phi\,=\,-\,\int_0^T\!\!\!\int_{\mbb{R}^2}\langle \vrho^{(1)}_\veps\rangle\, \d_t\phi\,-\,\int_{\mbb{R}^2}\langle \vrho^{(1)}_{0,\veps}\rangle\, \phi(0,\cdot )\,. 
\end{align*}

In addition, the convergence of the convective term has been performed in \eqref{eq:limit_T1-1}. As for other terms, we can argue as in Section \ref{ss:limit}.
Hence, letting $\varepsilon \rightarrow 0^+$ in the equation above, we get
\begin{align*}
&-\int_0^T\!\!\!\int_{\R^2} \left(\vec{U}^{h}\cdot \d_{t}\nabla_{h}^{\perp} \phi+ \vec{U}^{h}\otimes \vec{U}^{h}:\nabla_{h}(\nabla_{h}^{\perp}\phi )+\vrho^{(1)}\, \d_t \phi \right)\, dx^h\, dt\\
&\qquad\qquad=-\int_0^T\!\!\!\int_{\R^2} \mu \nabla_{h}\vec{U}^{h}:\nabla_{h}(\nabla_{h}^{\perp}\phi ) \, dx^h\, dt+\int_{\R^2}\left(\lan\vec{u}_{0}^{h}\ran\cdot \nabla _{h}^{\perp}\phi (0,\cdot )+
\lan \vrho^{(1)}_{0}\ran\phi (0,\cdot )\right) \, dx^h\, ,
\end{align*}
which is the weak formulation of equation \eqref{eq_lim:QG}. In the end, also Theorem \ref{th:m=1} is proved.

\appendix

\section{Appendix -- A few tools from Littlewood-Paley theory} \label{app:LP}

Let us present some tools from Littlewood-Paley theory, which we have exploited in our analysis.
We refer \tsl{e.g.} to Chapter 2 of \cite{B-C-D} for details.
For simplicity of exposition, we deal with the $\R^d$ case, with $d\geq1$; however, the whole construction can be adapted also to the $d$-dimensional torus $\TT^d$, and to the ``hybrid'' case $\R^{d_1}\times\TT^{d_2}$.

First of all, we introduce the so-called \emph{Littlewood-Paley decomposition}.
We fix a smooth radial function $\chi$ such that $\Supp\chi\subset B(0,2)$, $\chi\equiv 1$ in a neighborhood of $B(0,1)$
and the map $r\mapsto\chi(r\,e)$ is non-increasing over $\R_+$ for all unitary vectors $e\in\R^d$.
Set $\varphi\left(\xi\right)=\chi\left(\xi\right)-\chi\left(2\xi\right)$ and $\vphi_j(\xi):=\vphi(2^{-j}\xi)$ for all $j\geq0$.
The dyadic blocks $(\Delta_j)_{j\in\Z}$ are defined by\footnote{We agree  that  $f(D)$ stands for 
the pseudo-differential operator $u\mapsto\mc{F}^{-1}[f(\xi)\,\what u(\xi)]$.} 
$$
\Delta_j\,:=\,0\quad\mbox{ if }\; j\leq-2,\qquad\Delta_{-1}\,:=\,\chi(D)\qquad\mbox{ and }\qquad
\Delta_j\,:=\,\varphi(2^{-j}D)\quad \mbox{ if }\;  j\geq0\,.
$$
For any $j\geq0$ fixed, we  also introduce the \emph{low frequency cut-off operator}
\begin{equation} \label{eq:S_j}
S_j\,:=\,\chi(2^{-j}D)\,=\,\sum_{k\leq j-1}\Delta_{k}\,.
\end{equation}
Note that $S_j$ is a convolution operator. More precisely, after defining
$$
K_0\,:=\,\mc F^{-1}\chi\qquad\qquad\mbox{ and }\qquad\qquad K_j(x)\,:=\,\mathcal{F}^{-1}[\chi (2^{-j}\cdot)] (x) = 2^{jd}K_0(2^j x)\,,
$$
for all $j\in\N$ and all tempered distributions $u\in\mc S'$ we have that $S_ju\,=\,K_j\,*\,u$.
Thus the $L^1$ norm of $K_j$ is independent of $j\geq0$. This implies that $S_j$ maps continuously $L^p$ into itself, for any $1 \leq p \leq +\infty$.

Moreover, the following property holds true: for any $u\in\mc{S}'$, then one has the equality $u=\sum_{j}\Delta_ju$ in the sense of $\mc{S}'$.
Let us also recall the so-called \emph{Bernstein inequalities}.
  \begin{lemma} \label{l:bern}
Let  $0<r<R$.   A constant $C$ exists so that, for any non-negative integer $k$, any couple $(p,q)$ 
in $[1,+\infty]^2$, with  $p\leq q$,  and any function $u\in L^p$,  we  have, for all $\lambda>0$,
$$
\displaylines{
{\Supp}\, \widehat u \subset   B(0,\lambda R)\quad
\Longrightarrow\quad
\|\nabla^k u\|_{L^q}\, \leq\,
 C^{k+1}\,\lambda^{k+d\left(\frac{1}{p}-\frac{1}{q}\right)}\,\|u\|_{L^p}\;;\cr
{\Supp}\, \widehat u \subset \{\xi\in\R^d\,:\, \lambda r\leq|\xi|\leq \lambda R\}
\quad\Longrightarrow\quad C^{-k-1}\,\lambda^k\|u\|_{L^p}\,
\leq\,
\|\nabla^k u\|_{L^p}\,
\leq\,
C^{k+1} \, \lambda^k\|u\|_{L^p}\,.
}$$
\end{lemma}   

By use of Littlewood-Paley decomposition, we can define the class of Besov spaces.
\begin{definition} \label{d:B}
  Let $s\in\R$ and $1\leq p,r\leq+\infty$. The \emph{non-homogeneous Besov space}
$B^{s}_{p,r}$ is defined as the subset of tempered distributions $u$ for which
$$
\|u\|_{B^{s}_{p,r}}\,:=\,
\left\|\left(2^{js}\,\|\Delta_ju\|_{L^p}\right)_{j\geq -1}\right\|_{\ell^r}\,<\,+\infty\,.
$$
\end{definition}

Besov spaces are interpolation spaces between Sobolev spaces. In fact, for any $k\in\N$ and~$p\in[1,+\infty]$
we have the chain of continuous embeddings $ B^k_{p,1}\hookrightarrow W^{k,p}\hookrightarrow B^k_{p,\infty}$,
which, in the case when $1<p<+\infty$, can be refined to
$B^k_{p, \min (p, 2)}\hookrightarrow W^{k,p}\hookrightarrow B^k_{p, \max(p, 2)}$.
In particular, for all $s\in\R$ we deduce that $B^s_{2,2}\equiv H^s$, with equivalence of norms:
\begin{equation} \label{eq:LP-Sob}
\|f\|_{H^s}\,\sim\,\left(\sum_{j\geq-1}2^{2 j s}\,\|\Delta_jf\|^2_{L^2}\right)^{\!\!1/2}\,.
\end{equation}
Observe that, from that equivalence, we easily get the following property: for any $f\in H^s$ and any $j\in \N$, one has  
\begin{equation} \label{est:sobolev}
\left\|\big(\Id-S_j\big)f\right\|_{H^\s}\,\leq\,C\,\|\nabla f\|_{H^{s-1}}\,2^{-j(s-\s)} \qquad \text{ for all }\quad \s\leq s\,,
\end{equation}
where $C>0$ is a ``universal'' constant, independent of $f$, $j$, $s$ and $\s$. This inequality has been repeatedly used in our computations.

%


\end{document}